\documentclass[10pt]{amsart}
\usepackage{geometry}
\usepackage{amsmath}
\usepackage{amsthm}
\usepackage{amssymb}
\usepackage{graphicx}
\usepackage[multiple]{footmisc}
\usepackage{multicol}
\usepackage{tikz}
\usetikzlibrary{arrows}

\usepackage[pdfauthor={Kameryn J Williams},
    pdftitle={Minimum models of second-order set theories},
    hidelinks 
]{hyperref}

\usepackage[backend=bibtex,style=alphabetic,maxbibnames=15,maxcitenames=6,dateabbrev=false]{biblatex}
\renewcommand{\UrlFont}{\sffamily\smaller} 
\renewbibmacro{in:}{\ifentrytype{article}{}{\printtext{\bibstring{in}\intitlepunct}}} 
\DeclareFieldFormat{url}{{\UrlFont\url{#1}}} 
\DeclareFieldFormat{urldate}{
  (version \thefield{urlday}\addspace%
  \mkbibmonth{\thefield{urlmonth}}\addspace%
  \thefield{urlyear}\isdot)}
\DeclareFieldFormat{eprint:arxiv}{
  \ifhyperref
    {\href{http://arxiv.org/abs/#1}{%
        arXiv\addcolon\nolinkurl{#1}}\iffieldundef{eprintclass}{}{\UrlFont{\mkbibbrackets{\thefield{eprintclass}}}}}
    {arXiv\addcolon\nolinkurl{#1}\iffieldundef{eprintclass}{}{\UrlFont{\mkbibbrackets{\thefield{eprintclass}}}}}}

\bibliography{../../ref}

\theoremstyle{plain}
\ifdefined\theorem \else \newtheorem{theorem}{Theorem} \fi
\ifdefined\maintheorem \else \newtheorem{maintheorem}[theorem]{Main Theorem} \fi
\ifdefined\theoremschema \else \newtheorem{theoremschema}[theorem]{Theorem Schema} \fi
\ifdefined\lemma \else \newtheorem{lemma}[theorem]{Lemma} \fi
\ifdefined\lemmaschema \else \newtheorem{lemmaschema}[theorem]{Lemma Schema} \fi
\ifdefined\proposition \else \newtheorem{proposition}[theorem]{Proposition} \fi
\ifdefined\corollary \else \newtheorem{corollary}[theorem]{Corollary} \fi
\ifdefined\fact \else \newtheorem{fact}[theorem]{Fact} \fi
\ifdefined\problem \else \newtheorem{problem}[theorem]{Problem} \fi
\ifdefined\conjecture \else \newtheorem{conjecture}[theorem]{Conjecture} \fi
\ifdefined\question \else \newtheorem{question}[theorem]{Question} \fi
\ifdefined\observation \else \newtheorem{observation}[theorem]{Observation} \fi
\newtheorem*{theorem*}{Theorem}
\newtheorem*{lemma*}{Lemma}
\newtheorem*{proposition*}{Proposition}
\newtheorem*{conjecture*}{Conjecture}
\newtheorem*{question*}{Question}
\newtheorem*{definition*}{Definition}
\newtheorem{sublemma}{Lemma}[theorem]
\theoremstyle{definition}
\ifdefined\definition \else \newtheorem{definition}[theorem]{Definition} \fi
\ifdefined\maindefinition \else \newtheorem{maindefinition}[theorem]{Main Definition} \fi
\theoremstyle{remark}
\ifdefined\remark \else \newtheorem{remark}[theorem]{Remark} \fi
\ifdefined\remarks \else \newtheorem{remarks}[theorem]{Remarks} \fi
\ifdefined\example \else \newtheorem{example}[theorem]{Example} \fi
\ifdefined\claim \else \newtheorem{claim}[theorem]{Claim} \fi
\ifdefined\exercise \else \newtheorem{exercise}[theorem]{Exercise} \fi
\newtheorem*{remark*}{Remark}
\ifdefined\acknowledgment \else \newtheorem*{acknowledgment}{Acknowledgment} \fi
\ifdefined\dedication \else \newtheorem*{dedication}{Dedicated} \fi
\ifdefined\case \else \newtheorem*{case}{Case} \fi

\newcounter{my_enumerate_counter}

\newcommand\comment[1]{}

\newcommand\Zsf{\mathsf{Z}}

\newcommand\Afrak{\mathfrak{A}}
\newcommand\Bfrak{\mathfrak{B}}

\newcommand\Mfrak{\mathfrak{M}}

\newcommand\Lcal{\mathcal{L}}

\newcommand\Pcal{\mathcal{P}}

\newcommand\Xcal{\mathcal{X}}
\newcommand\Ycal{\mathcal{Y}}
\newcommand\Zcal{\mathcal{Z}}

\renewcommand\Bbb{\mathbb{B}}

\newcommand\Pbb{\mathbb{P}}
\newcommand\Qbb{\mathbb{Q}}

\newcommand{\dom}{\operatorname{dom}}

\ifdefined\alt \else
  \newcommand{\alt}{\operatorname{alt}}
\fi

\newcommand{\forces}{\Vdash}

\newcommand\cat{{}^\smallfrown}

\renewcommand{\diamond}{\diamondsuit}

\newcommand\axiom{\mathsf}

\newcommand\KP{\axiom{KP}}
\newcommand\KPU{\axiom{KPU}}
\newcommand\ZF{\axiom{ZF}}
\newcommand\ZFC{\axiom{ZFC}}
\newcommand\ZFm{\axiom{ZF}^-}
\newcommand\ZFCm{\axiom{ZFC}^-}

\newcommand\GB{\axiom{GB}}
\newcommand\GBC{\axiom{GBC}}

\newcommand\ETR{\axiom{ETR}}

\newcommand\DCA{\Delta^1_1\text{-}\axiom{CA}}

\newcommand\PCA{\Pi_1^1\text{-}\axiom{CA}}
\newcommand\PnCA[1]{\Pi_{#1}^1\text{-}\axiom{CA}}

\newcommand\KM{\axiom{KM}}

\newcommand\ZFmi{\ZFm_{\mathrm I}}
\newcommand\ZFCmi{\ZFCm_{\mathrm I}}

\newcommand\GC{\axiom{GC}}

\newcommand\AC{\axiom{AC}}

\newcommand\PA{\axiom{PA}}

\newcommand\RCA{\axiom{RCA}}
\newcommand\WKL{\axiom{WKL}}
\newcommand\ACA{\axiom{ACA}}
\newcommand\ATR{\axiom{ATR}}

\newcommand\class{\mathrm}

\newcommand\HOD{\class{HOD}}
\newcommand\Ord{\class{Ord}}

\newcommand{\seq}[1]{\left\langle #1 \right\rangle}

\renewcommand{\epsilon}{\varepsilon}

\newcommand\mand{\textrm{ and }}

\newcommand\length{\operatorname{len}}
\newcommand\rest{\upharpoonright}

\newcommand\powerset{\Pcal}

\newcommand\tc{\operatorname{tc}}

\newcommand\omegaoneck{{\omega_1^{\mathrm{CK}}}}
\newcommand\Hyp{\operatorname{Hyp}}

\newcommand{\proves}{\vdash}

\newcommand{\godel}[1]{\ulcorner#1\urcorner}

\newcommand\Sk{\operatorname{Sk}}

\newcommand\Con{\operatorname{Con}}

\newcommand\Def{\operatorname{Def}}
\ifdefined\Form \else
  \newcommand\Form{\axiom{Form}}
\fi

\newcommand{\impl}{\Rightarrow}
\renewcommand{\iff}{\Leftrightarrow}
\renewcommand{\phi}{\varphi}

\newcommand\Tr{\mathrm{Tr}}

\newcommand\KMC{\axiom{KMC}}

\title{Minimum models of second-order set theories}
\DeclareRobustCommand{\okina}{%
  \raisebox{\dimexpr\fontcharht\font`A-\height}{%
    \scalebox{0.8}{`}%
  }%
}
\author{Kameryn J.\ Williams}
\address[K.~J.~Williams]{
University of Hawai\okina{}i at M\=anoa \\
Department of Mathematics \\
2565 McCarthy Mall, Keller 401A \\
Honolulu, HI  96822\\
USA
}
\email{kamerynw@hawaii.edu}
\urladdr{http://kamerynjw.net}

\thanks{I am grateful to the anonymous referee for their many helpful comments. The present article is substantially improved thanks to their time and care in reviewing it. I also want to thank Joel David Hamkins, under whose supervision I wrote my dissertation, written while this paper was under review and consisting in part of the material herein.}

\subjclass[2010]{Primary 03E70; Secondary 03C62}

\keywords{Kelley--Morse, G\"odel--Bernays, elementary transfinite recursion, minimum model, second-order set theory}

\begin{document}

\begin{abstract}
In this article I investigate the phenomenon of minimum models of second-order set theories, focusing on Kelley--Morse set theory $\KM$, G\"odel--Bernays set theory $\GB$, and $\GB$ augmented with the principle of Elementary Transfinite Recursion. The main results are the following. $(1)$ A countable model of $\ZFC$ has a minimum $\GBC$-realization if and only if it admits a parametrically definable global well-order. $(2)$ Countable models of $\GBC$ admit minimal extensions with the same sets. $(3)$ There is no minimum transitive model of $\KM$. $(4)$ There is a minimum $\beta$-model of $\GB + \ETR$.
The main question left unanswered by this article is whether there is a minimum transitive model of $\GB + \ETR$.
\end{abstract}

\maketitle

\section{Introduction}

The jumping-off point for this article is the following well-known theorem.

\begin{theorem*}[Shepherdson \cite{shepherdson1953}, Cohen \cite{cohen1963}]
If there is a transitive model of $\ZF$ then there is a minimum transitive model of $\ZF$. That is, there is a transitive model $M$ of $\ZF$ so that if $N$ is any transitive model of $\ZF$ then $M \subseteq N$. Indeed, this minimum model also satisfies $\AC$ and $V = L$.
\end{theorem*}

From a modern perspective, this result is not difficult to prove. This minimum model is $L_\alpha$ where $\alpha$ is the least height of a transitive model of $\ZF$. The absoluteness of $L$ gives that $L_\alpha$ is contained in every transitive model of $\ZF$.\footnote{The anonymous referee noted that the Shepherdson-Cohen minimum model $L_\alpha$ is also the minimum model of $\ZF$ in the following sense: it is, up to isomorphism, the unique model of $\ZF$ which can be isomorphically embedded into every model of $\ZF$. I reproduce their proof of this fact here. If $N$ is well-founded, then by the Shepherdson--Cohen theorem $L_\alpha$ embeds into a transitive submodel of $N$. In this case, the embedding is moreover $\Delta_0$-elementary. If $N$ is not well-founded, it has a countable elementary submodel $N_0$ which must also be ill-founded. By a theorem of Hamkins \cite[corollary 29]{hamkins2013} there is therefore an embedding of $L_\alpha$ into $N_0$, hence an embedding of $L_\alpha$ into $N$. That $L_\alpha$ is the unique such model follows from the Shepherdson--Cohen theorem combined with the fact that any model which embeds into a well-founded model must also well-founded.}

This proof can be generalized to get analogous results for many weakenings or strengthenings $T$ of $\ZF$. What we need from $T$ is that it has a transitive model, is strong enough to build the inner model $L$, and is absolute to $L$, meaning that $M \models T$ implies $L^M \models T$. On the weaker side, $\KP$ satisfies the conditions---see \cite[chapter II]{barwise1975}. The minimum transitive model of $\KP$ is $L_\omega = V_\omega$ and the minimum transitive model of $\KP$ $+$ Infinity is $L_{\omegaoneck}$, where $\omegaoneck$ is the Church--Kleene ordinal. On the stronger side, we get minimum transitive models of $\ZFC$ plus small large cardinals. There is, for instance, a minimum transitive model of $\ZFC$ $+$ there is a Mahlo cardinal. It is not difficult to see, however, that this cannot be extended too far up the large cardinal hierarchy. 

There is no minimum transitive model of any $T$ extending $\ZFC$ which proves there is a measurable cardinal. This is because by a well-known theorem of Scott \cite{scott1961}, if $N \models \ZFC$ has a measurable cardinal then there is an elementary embedding $j : N \to M$ to an inner model $M$, with the measure not in $M$.\footnote{For a full proof, consult \cite[Lemma 17.9 $(ii)$]{jech:book}.} So any transitive model of such $T$ contains a strictly smaller transitive model of $T$.

Moreover observe that Scott's theorem can be internalized to any model of $\ZFC$ with a measurable cardinal, whether or not it is well-founded. This yields that any $M \models \ZFC$ $+$ ``there is measurable cardinal'' has a proper transitive submodel which satisfies the same first-order theory. Define the pre-order $\vartriangleleft_\mathsf{end}$ on models of $\ZFC$ as $M \vartriangleleft_\mathsf{end} N$ if  $M$ embeds as a transitive submodel of $N$. Or to use language that will be introduced in the next section, $M \vartriangleleft_\mathsf{end} N$ if and only if $N$ \emph{end-extends} an isomorphic copy of $M$. So this same argument shows that if $T$ is any extension of $\ZFC$ which proves there is a measurable cardinal, then there is no $\vartriangleleft_\mathsf{end}$-minimal model of $T$.\footnote{I thank the referee for pointing out this extension of the argument.}

In contrast, results from inner model theory show that if there is a transitive model which thinks the ordinal $\kappa$ is measurable, then there is a minimum such model. This extends to strong cardinals, Woodin cardinals, etc.\footnote{Briefly: It is well-known that if $\kappa$ is measurable there is a minimum inner model which thinks $\kappa$ is measurable, of the form $L[U]$ where $U$ is a normal measure on $\kappa$. See \cite[section 20]{kanamori-book} for a proof. From this it follows that if $\alpha$ is the least height of a model of $\ZFC$ which thinks $\kappa$ is measurable then a model of the form $L_\alpha[U]$ is the minimum transitive model which thinks that $\kappa$ is measurable. A similar argument can be made for other large cardinals with canonical inner models. The reader is advised to consult \cite{zeman:InnerModels} for the construction up to a strong cardinal, and \cite{mitchell-steel:FineStructure} for the argument for Woodin cardinals.}

A second direction the Shepherdson--Cohen result could be generalized is from first-order set theory to second-order set theory. Given some second-order set theory, does it have a minimum transitive model? A natural first theory to look at here is $\GB$, G\"odel--Bernays set theory. This was already answered by  Shepherdson in the positive; see theorem \ref{thm:min-gb}.

While the existence of a minimum transitive model of $\GB$ is already settled, I also look the situation where we fix the sets of our model but allow the classes to vary. In this context, it is most interesting to consider $\GBC$, G\"odel--Bernays set theory with Global Choice.

\begin{maintheorem} \label{main:min-gbc-rlzn}
A countable model $M \models \ZFC$ has a minimum $\GBC$-realization if and only if $M$ has a parametrically definable global well-ordering.
\end{maintheorem}

\begin{maintheorem} \label{main:min-gbc-ext}
If $(M,\Xcal) \models \GBC$ is countable, then $\{ \Ycal : (M,\Ycal) \models \GBC$ and $\Xcal \subsetneq \Ycal\}$ has minimal elements.
\end{maintheorem}

In contrast to the Shepherdson result, stronger second-order set theories lack minimum transitive models.

\begin{maintheorem} \label{min-trans-km}
There is no minimum transitive model of $\KM$. Nor is there a minimum transitive model of $\GB + \PCA$, nor any computably axiomatizable extension thereof.
\end{maintheorem}

I also consider theories intermediate between $\GB$ and $\GB + \PCA$. The most important such theory I will consider is $\GB + \ETR$, where $\ETR$ is the principle of Elementary Transfinite Recursion. And I consider other theories obtained by weakening $\ETR$. My main result here concerns $\beta$-models of $\GB + \ETR$. A transitive model of second-order set theory is a \emph{$\beta$-model} if it is correct about well-foundedness for class relations. This gives a more restricted class of models.

\begin{maintheorem} \label{min-beta-etr}
There is a minimum $\beta$-model of $\GB + \ETR$, if there is any $\beta$-model of $\GB + \ETR$.
\end{maintheorem}

Combined with the results that $\GB$ and $\KM$ have minimum $\beta$-models, due respectively to Shepherdson \cite{shepherdson1953} and Marek and Mostowski \cite{marek-mostowski1975},\footnote{See, respectively, corollary \ref{cor:least-beta-gbc} and theorem \ref{thm:least-beta-km}.} 
we have the following state of affairs.

\begin{center}
\begin{tabular}{c|c c}
 & minimum transitive model? & minimum $\beta$-model? \\
\hline
$\GB$ & Y & Y \\
$\GB + \ETR$ & ? & Y \\
$\KM$  & N & Y
\end{tabular}
\end{center}

Left open by this article is filling in the question mark in the table: does $\GB + \ETR$ have a minimum transitive model? While I do not answer this question, I do show that weak enough fragments of $\GB + \ETR$ do have minimum transitive models; in particular, $\GB + \ETR_\Ord$ has a minimum transitive model.\footnote{$\ETR_\Ord$ is the restriction of $\ETR$ to recursions along $\Ord$. See definition~\ref{def:etr-gamma}.}

The structure of this article is as follows. I first discuss some preliminary definitions and results, followed by a look at some antecedents for my work in the literature, and parallels to work in second-order arithmetic. The remainder of the paper is organized by the strength of the theories considered. These three sections are largely independent, and may be read by the reader in any order. First comes the section about $\GB$. I present proofs of the above-mentioned Shepherdson results and prove main theorems \ref{main:min-gbc-rlzn} and \ref{main:min-gbc-ext}. The next section is devoted to $\ETR$ and its fragments. I prove main theorem \ref{min-beta-etr} and show that there are minimum transitive models for $\GB$ plus weak enough fragments of $\ETR$. I end the section by separating these fragments of $\ETR$. The final section of the paper is about $\KM$ and other strong second-order set theories.
I prove main theorem \ref{min-trans-km} and give a proof of the Marek--Mostowski result that there is a minimum $\beta$-model of $\KM$. I also show that no countable transitive model of $\ZF$ has a minimum $\KM$-realization. 

\section{Preliminaries} \label{sec:prelim}

The second-order set theories considered in this article are formulated so that their models consist of two sorts, the \emph{sets} of or \emph{first-order part} of the model and the \emph{classes} of or \emph{second-order part} of the model.
All second-order set theories considered in this paper will include an axiom asserting that extensionality holds for classes. This ensures that any model of these theories is isomorphic to one whose second-order part consists of subsets of the first-order part. Without loss we can consider only such models. Formally, models of second-order set theories are of the form $(M,\Xcal; \in^{(M,\Xcal)})$ where $M$ is the first-order part, $\Xcal$ is the second-order part, $\in^{(M,\Xcal)}$ is the membership relation for the model between sets and sets or between sets and classes. For simplicity, I will write $(M,\Xcal)$, suppressing writing the membership relation. I will use $\in^M$, or simply $\in$ if it is unambiguous, to refer to the membership relation of $M$. In general, when talking about a multi-sorted structure I will separate the sorts and the relations, functions, and constants with a semicolon.

A model $(M,\Xcal)$ of second-order set theory is {\em transitive} if $M$ and $\Xcal$ are transitive and its membership relation is the true $\in$. By the convention that $\Xcal$ consists of subsets of $M$, for the theories considered in this article it is enough just to ask for $M$ itself to be transitive.

\begin{definition}
Given a second-order set theory $T$, the {\em minimum transitive model of $T$}, provided it exists, is a transitive model $(M,\Xcal)$ of $T$ so that if $(N,\Ycal)$ is any transitive model of $T$ then $M \subseteq N$ and $\Xcal \subseteq \Ycal$. For the theories we consider, this is equivalent to requiring that $\Xcal \subseteq \Ycal$.
\end{definition}

If $M$ is a first-order model of set theory and $T$ is a second-order set theory, say that $\Xcal \subseteq \powerset(M)$ is a $T$-{\em realization for} $M$ if $(M,\Xcal) \models T$. If there is some $T$-realization for $M$ then we say that $M$ is $T$-{\em realizable}. The {\em minimum $T$-realization for $M$}, provided it exists, is $\Xcal \subseteq \powerset(M)$ so that $(M,\Xcal) \models T$ and for any $\Ycal \subseteq \powerset(M)$ if $(M,\Ycal) \models T$ then $\Xcal \subseteq \Ycal$. That is, $\Xcal$ is the minimum $T$-realization for $M$ if it is the minimum element in the poset consisting of all $T$-realizations for $M$ ordered by the subset relation. 

We can also talk about minimal $T$-realizations, minimal transitive models of $T$, and so forth. Recall the distinction between minimal and minimum elements. Given a partial order $(P,<)$, a \emph{minimal} element $m$ is one so that there is no $p < m$ in $P$. A \emph{minimum} element $m$ has the stronger property that $m \le p$ for all $p \in P$. So, for example, a minimal transitive model of $\ZF$ is one which does not contain any strictly smaller transitive model, while the minimum transitive model of $\ZF$ is contained inside all transitive models.\footnote{It is unfortunate that mathematicians are not always careful about the distinction; for instance, one sometimes hears the Shepherdson--Cohen minimum transitive model called the \emph{minimal} model of $\ZF$, even though it satisfies the stronger property. And observe that the existence of a minimal transitive model of $\ZF$ is trivially proven from the assumption of the existence of a transitive model of $\ZF$, simply by considering any model of minimum ordinal height. On the other hand, the Shepherdson and Cohen proofs that there is a minimum transitive model use considerable technology.}

Observe that if $\Xcal$ is the minimum $T$-realization for $M$, then $\Xcal = \bigcap \{ \Ycal : (M,\Ycal) \models T \}$. This is the motivation for the following definition.

\begin{definition} \label{def:xcalt}
Let $T$ be a second-order set theory and let $M$ be a model of $\ZF$. Set 
\[
\Xcal_T(M) = \bigcap \{ \Ycal : (M,\Ycal) \models T \}.
\]
If the context is clear I will just write $\Xcal_T$.
\end{definition}

We can then phrase the question of whether $M$ has a minimum $T$-realization as: is it the case that $(M,\Xcal_T) \models T$? In case the answer is no, the question is then: what axioms are satisfied by $(M,\Xcal_T)$?

A transitive model $(M,\Xcal)$ of second-order set theory is a {\em $\beta$-model} if it is correct about well-foundedness for class relations. That is, if $R \in \Xcal$ and $(M,\Xcal) \models \text{``}R \text{ is well-founded''}$ then $R$ really is well-founded. 
Not every transitive model is a  $\beta$-model, even if we restrict to models satisfying a strong second-order set theory. This is in contrast to the well-known fact that well-foundedness is absolute for transitive models of $\ZFC$. For models of $\ZFC$, there are two characterizations of well-foundedness: a $\Pi_1$ characterization, namely that every nonempty subset of $\dom R$ has an $R$-minimal element; and a $\Sigma_1$ characterization, namely there is a ranking function from $R$ to $\Ord$.

In the second-order context, we still have access to the $\Pi^0_1$ characterization.\footnote{To clarify, the superscript $0$ and $1$ refer to whether a formula includes quantification over classes, rather than the arithmetical/analytical hierarchies from descriptive set theory. The $0$ is for first-order quantification; for instance, a $\Pi^0_1$ formula has a single unbounded universal set quantifier, with all other quantifiers bounded. The $1$ is for second-order quantification; for instance, a $\Sigma^1_1$ formula is of the form $\exists X \phi(X,P)$ where $\phi$ is $\Pi^0_k$ for some $k$.}
The theories we consider will prove that a class relation $R$ has that every nonempty sub\emph{class} of its domain has a minimal element if and only if every nonempty sub\emph{set} of its domain has a minimal element.\footnote{See observation \ref{obs:inner-beta-models}.}
So in fact it takes only a single unbounded first-order universal quantifier to assert that $R$ is well-founded. 

As such, well-foundedness is downward absolute to transitive models of second-order set theory; if $R \in \Xcal$ is well-founded in $V$ then $(M,\Xcal)$ thinks $R$ is well-founded. However, well-foundedness for class relations is not upward absolute for transitive models of second-order set theory. The trouble is that class relations can have rank $> \Ord$ so we cannot use the $\Delta^0_0$ definition of the von Neumann ordinals to get a $\Sigma^0_1$ characterization for well-foundedness. (Of course, well-foundedness for {\em set} relations is upward absolute for transitive models.)

Similar to the definitions of the minimum transitive model of $T$, a $T$-realization for $M$, a $T$-realizable model, and the minimum $T$-realization for $M$ we can define the {\em minimum $\beta$-model of $T$}, a {\em $\beta$-$T$-realization for $M$}, a {\em $\beta$-$T$-realizable} model, and the {\em minimum $\beta$-$T$-realization for $M$}. For instance, $\Xcal \subseteq \powerset(M)$ is a {\em $\beta$-$T$-realization for $M$} if $(M,\Xcal) \models T$ is a $\beta$-model. I will also use $\Xcal_{\beta,T}(M)$ to refer to the intersection of $\beta$-$T$-realizations for $M$.

The convention in this paper when talking about formulae in the language of set theory is to use capital letters for classes and class variables and lowercase letters for sets and set variables. For example, the formula $\forall X \exists y\ y \in X$ gives the (false) proposition that every class has some set as a member.

\begin{definition}
Kelley-Morse set theory $\KM$ is the second-order set theory consisting of the following axioms:
\begin{itemize}
\item $\ZF$ for the first-order part;
\item Extensionality for classes;
\item Class Replacement, asserting that the image of any set under a class function is a set;
\item The Comprehension axiom schema, namely the axioms $\forall \bar P \exists A \forall x \ x \in A \iff \phi(x,\bar P)$, for every second-order formula $\phi(x,\bar P)$. That is, $\phi$ can include quantifier over sets or over classes.
\end{itemize}
$\KM$ can be augmented with Global Choice $\GC$, asserting that there is a bijection between $V$ and $\Ord$, to get $\KMC$. In general, my convention will be that if $\mathsf{BAR}$ is a second-order set theory I will use $\mathsf{BARC}$ to mean $\mathsf{BAR} + \GC$. 
\end{definition}

G\"odel--Bernays set theory $\GB$ is axiomatized similarly, except that the Comprehension schema is weakened. Its axioms are $\ZF$ for the first-order part, Extensionality for classes, Class Replacement, and the Elementary Comprehension schema. This is a subschema of the  Comprehension schema where we restrict to only formulae which do not include class quantifiers. Following the convention for theories with Global Choice, $\GBC$ is $\GB + \GC$. 
We can consider other restrictions of the Comprehension schema to get a hierarchy of theories ranging in strength from $\GB$ to $\KM$. For a natural number $k$, let $\PnCA k$ denote the restriction of Class Comprehension to the $\Pi^1_k$-formulae. Then $\GB + \PnCA 0$ is just $\GB$ and $\KM$ is $\GB + \bigcup_{k \in \omega} \PnCA k$. Levels of this hierarchy are separated in terms of consistency strength. That is, $\GB + \PnCA{k+1}$ proves the consistency of $\GB + \PnCA k$. 
In this article we will also see $\DCA$, Comprehension for $\Delta^1_1$-properties.\footnote{To be clear, instances of $\DCA$ have the form 
\[
\forall \bar P ([\forall x\ \forall Y \phi(x,Y,\bar P) \iff \exists Y\ \psi(x,Y,\bar P)] \impl \exists A\ \forall x\ [x \in A \iff \forall Y \phi(x,Y,\bar P)]),
\]
where $\phi$ and $\psi$ have no class quantifiers.}

Observe that Class Replacement together with Elementary Comprehension yield that Separation and Replacement hold allowing classes as parameters.

Intermediate between $\PnCA 0$ and $\PCA$ is the principle $\ETR$ of Elementary Transfinite Recursion. This schema asserts that if $\phi$ is a first-order formula (possibly with class parameters) and $R$ is a well-founded class relation then there is a solution to the recursion of $\phi$ along $R$. Formally, let $\phi(x,Y,A)$ be a first-order formula, possibly with a class parameter $A$ and let $R$ be a well-founded class relation. Denote by $<_R$ the transitive closure of $R$. The instance of $\ETR$ for $\phi$ and $R$ asserts that there is a class $S \subseteq \dom R \times V$ which satisfies
\[
S_r = \{ x : \phi(x,S \rest r, A) \}
\]
for all $r \in \dom R$. Here, $S_r = \{ x : (r,x) \in S \}$ denotes the $r$-th slice of $S$ and
\[
S \rest r = S \cap \left(\{ r' \in \dom R : r' <_R r\} \times V\right)
\]
is the partial solution below $r$. It is equivalent to define $\ETR$ by restricting the well-founded relations allowed to well-founded partial orders, or to well-founded trees, or to well-orders; see \cite[lemma 7]{gitman-hamkins2017}.

One well-known class recursion is the Tarskian definition of truth. This recursion takes place over a proper class tree of height $\omega$. Thus, $\ETR$ asserts the existence of a first-order truth predicate and hence proves $\Con(\ZF)$, as restricting the truth predicate to parameter-free formulae gives a consistent completion of $\ZF$.\footnote{If the model is $\omega$-nonstandard, meaning that its $\omega$ is ill-founded, things are a little more delicate. The model will still think that $\ZF$ is consistent because what it thinks is $\ZF$, nonstandard instances of the axiom schemata and all, will be contained in what it thinks is the first-order theory of the model, which it necessarily thinks is complete and consistent.} This shows that $\ETR$ exceeds $\GB$ in consistency strength. 
For an upper bound, one can use $\PCA$ to show that there is no least failure of a potential solution to an elementary recursion and from that conclude that a solution must exist. Indeed, $\GB + \ETR + \DCA$ proves the consistency of $\GB + \ETR$ \cite[theorem 33]{sato2014}.
In particular, $\GB + \PCA$ proves the consistency of $\GB + \ETR$. (And we can add Global Choice to both theories here and the result remains true.)

A useful fact is that Elementary Transfinite Recursion is equivalent to having solutions to a certain class of recursions. See \cite[Corollary 61]{fujimoto2012} and \cite[Theorem 8]{gitman-hamkins2017}.

\begin{theorem} \label{thm:etr-itr}
Over $\GB$, Elementary Transfinite Recursion is equivalent to the assertion that for every class $A$ and every class well-order $\Gamma$ there is a $\Gamma$-iterated first-order truth predicate relative to $A$.
\end{theorem}

A $\Gamma$-iterated truth predicate is a class $T$ of triples $(a,\phi,p)$ where $(a,\phi,p) \in T$ intuitively means that $\phi(p)$ is true at level $a \in \dom(\Gamma)$. Here, $\phi$ is in the language with $\in$ and a predicate symbol $\hat T$ for $T$. Formally, this is defined by a modified form of the Tarskian recursion, with an extra clause in the definition asserting that $(a, \godel{\hat T(x, y, z)}, \seq{b,\phi,p}) \in T$ if and only if $b <_\Gamma a$ and $(b,\phi,p) \in T$. A truth predicate relative to $A$ is for the language expanded with a predicate symbol $\hat A$ for $A$, with an extra clause in the definition asserting that $(a,\godel{\hat A(x)}, p) \in T$ if and only if $p \in A$.

It can be proven from $\GB$ that any two would-be $\Gamma$-iterated truth predicates relative to $A$ must be the same class: By Elementary Comprehension we can form the class of indices where they disagree, and so if they disagree they must disagree at a minimal point, from which we can derive a contradiction. I will use $\Tr_\Gamma(A)$ to refer to the unique $\Gamma$-iterated truth predicate relative to $A$. Iterated truth predicates not relative to any class will be denoted $\Tr_\Gamma$. As a special case, $\Tr(A) = \Tr_1(A)$ is the ordinary truth predicate relative to $A$.

I include a brief proof sketch for theorem \ref{thm:etr-itr} below, in order to illustrate that we in fact get a stronger result. For a full proof see the above cited literature.

\begin{proof}[Proof sketch of theorem \ref{thm:etr-itr}]
$(\Rightarrow)$ Construct $\Tr_\Gamma(A)$ by means of a recursion of rank $\omega \cdot \Gamma$.

$(\Leftarrow)$ Given $\Tr_\Gamma(A)$ one can construct a solution to a recursion of $\phi$ (with parameter $A$) along a relation of rank $\Gamma$.
\end{proof}

As an immediate corollary of this proof we get the following.

\begin{definition} \label{def:etr-gamma}
For a class well-order $\Gamma$, let $\ETR_\Gamma$ denote the principle of Elementary Transfinite Recursion restricted to relations of rank $\le \Gamma$ and let $\ETR_{<\Gamma}$ denote the restriction of Elementary Transfinite Recursion to relations of rank $< \Gamma$. 
\end{definition}

\begin{corollary} \label{cor:etr-iff-itr}
If $\omega^\omega \le \Gamma$ then $\ETR_\Gamma$ is equivalent over $\GB$ to the existence of $\Gamma$-iterated truth predicates relative to any class.
\end{corollary}

\begin{proof}
The assumption on $\Gamma$ implies that $\omega \cdot \Gamma < \Gamma + \Gamma$. Because $\ETR_\Gamma$ implies $\ETR_{\Gamma + \Gamma}$ (and $\ETR_{\Gamma \cdot n}$ for all standard $n$), this gives that we can do the necessary recursions of rank $\omega \cdot \Gamma$.
\end{proof}

It must be clarified just what it means for a model to satisfy $\ETR_\Gamma$. If $\Gamma \in \Xcal$ then it makes sense to 
ask whether $(M,\Xcal)$ satisfies $\ETR_\Gamma$---simply use $\Gamma$ as a parameter. However, this approach flounders upon the questions of whether there is a minimum $\beta$-model of or a minimum transitive model of $\GB + \ETR_\Gamma$. How should $\Gamma$ be interpreted if we do not already have a model in mind?

To handle this context, suppose $\Gamma$ is a name for a class well-order, meaning $\Gamma$ is defined by some formula which $\GB$ proves defines a well-order. For example, $\Gamma$ could be $\omega^\omega$, $\omega_1$, $\Ord$, or some other definable well-order. I will slightly abuse notation and use $\ETR_\Gamma$ in this context, referring to the theory in the language with $\in$ as its only non-logical symbol. And $\ETR_{<\Gamma}$ is similarly used in this context. One natural instance of this is $\ETR_\Ord$, asserting that recursions of rank $\le \Ord$ have solutions.\footnote{To briefly attempt to justify the naturalness of $\ETR_\Ord$, Gitman, Hamkins, Holy, Schlict, and myself showed \cite{GHHSW2017} that $\ETR_\Ord$ exactly captures the strength of the forcing theorem for class forcings, along with several other principles.}
Different models will have different $\Ord$s, but since $\GB$ proves that $\Ord$ is a well-order, it always is sensible to ask whether $(M,\Xcal)$ satisfies $\ETR_\Ord$. 

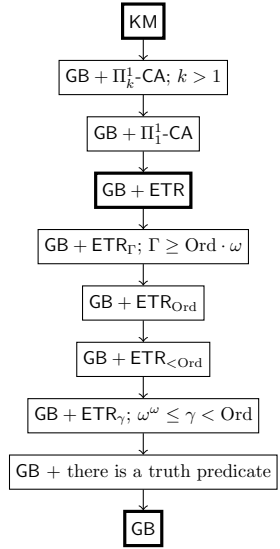
\smallskip

\begin{center}
\begin{tikzpicture}[best/.style={draw,very thick,rectangle,scale=.7,minimum height=6.5mm,fill=white},
good/.style={draw,rectangle,scale=.7,minimum height=6.5mm,fill=white},
scale=.3]
\draw (0:0) node[best] (GBC) {$\GB$}
 ++(90:2.5) node[good] (GBCtr) {$\GB\ +$ there is a truth predicate}
 ++(90:2.5) node[good] (ETRgamma) {$\GB + \ETR_\gamma$; $\omega^\omega \le \gamma < \Ord$}
 ++(90:2.5) node[good] (ETRlord) {$\GB + \ETR_{<\Ord}$}
 ++(90:2.5) node[good] (ETRord) {$\GB + \ETR_\Ord$}
 ++(90:2.5) node[good] (ETRbiggamma) {$\GB + \ETR_\Gamma$; $\Gamma \ge \Ord \cdot \omega$}
 ++(90:2.5) node[best] (ETR) {$\GB + \ETR$}
 ++(90:2.5) node[good] (PCA) {$\GB + \PCA$}
 ++(90:2.5) node[good] (PkCA) {$\GB + \PnCA{k}$; $k > 1$}
 ++(90:2.5) node[best] (KM) {$\KM$};
\draw[<-]
 (GBC) edge (GBCtr) (GBCtr) edge (ETRgamma) (ETRgamma) edge (ETRlord) (ETRlord) edge (ETRord) (ETRord) edge (ETRbiggamma) (ETRbiggamma) edge (ETR) (ETR) edge (PCA) (PCA) edge (PkCA) (PkCA) edge (KM);
\end{tikzpicture}

Some second-order set theories from $\GB$ to $\KM$, separated by consistency strength.
\end{center}

\subsection{Models of second-order set theories}

In the interest of completeness of presentation, I will summarize some well-known facts about how we construct models of these second-order set theories.

Let $M \models \ZF$. Say that $A \subseteq M$ is {\em strongly amenable (to $M$)} if $(M,A) \models \ZF(A)$, where $\ZF(A)$ is the expansion of $\ZF$ obtained by allowing a  predicate symbol for $A$ in the axiom schemata of $\ZF$. If $A$ is strongly amenable then $A$ is {\em amenable}, meaning that for $A \cap x$ is a set in $M$ for any $x \in M$.
Observe that every class in a model of $\GB$ is strongly amenable to the first-order part. 

The following observation is key in constructing models of $\GB$ and $\GBC$.

\begin{observation} \label{obs:gbc}
Let $M$ be a model of $\ZF$ and suppose $A \subseteq M$ is strongly amenable. Then $\Def(M;A)$ is a $\GB$-realization for $M$, where $\Def(M;A)$ is the collection of classes of $M$ first-order definable from $A$ and set parameters. If a global well-order of $M$ is definable from $A$, then $\Def(M;A)$ is a $\GBC$-realization for $M$.
\end{observation}

In particular, $\Def(M) = \Def(M;\emptyset)$ is always a $\GB$-realization for $M$. So any model of $\ZF$ is $\GB$-realizable. If $M$ moreover has a definable (from parameters) global well-order then $\Def(M)$ is a $\GBC$-realization for $M$. Note that having a definable global well-order is equivalent to satisfying the axiom $\exists x \ V = \HOD(\{x\})$, where $\HOD(\{x\})$ denotes the (definable) class of all sets hereditarily definable from ordinal parameters and $x$. (Having a global well-order definable without parameters is equivalent to $V = \HOD$.) Therefore, any model of of $\ZFC + \exists x\ V = \HOD(\{x\})$ is $\GBC$-realizable. 

We can do better. Any countable model of $\ZFC$ has a $\GBC$-realization, even if it does not have a definable global well-order. This can be seen by a class forcing argument, which was originally noticed by several mathematicians, among them Cohen, Felgner, and Solovay. Of these, only Felgner published the argument \cite{felgner1971,felgner1976}. Gaifman later found a forcing-free model-theoretic proof \cite{gaifman1975}. To sketch the argument: If $C$ is a Cohen-generic subclass of $\Ord$, then by density $C$ codes a strongly amenable global well-order. If $M \models \ZFC$ is countable we can find such $C \in V$ as we only need to meet countably many dense definable classes of $M$. This gives us $(M,\Def(M;C)) \models \GBC$ in $V$. As a consequence, $\GBC$ is conservative over $\ZFC$. That is, any first-order theorem of $\GBC$ is a theorem of $\ZFC$. In particular, $\GBC \not \proves \Con(\ZFC)$, supposing $\GBC$ is consistent.

This construction does not generalize to uncountable models of $\ZFC$. Say that a model of set theory is {\em rather classless} if its only amenable classes are its classes definable from set parameters. Note that the height of a rather classless model must have uncountable cofinality. Any $\omega$-sequence cofinal in the ordinals of a model will be amenable because its initial segments are finite. 

\begin{definition}
Let $M$ be a model of set theory. An extension $N \supseteq M$ is an \emph{end-extension} if $M$ is a transitive subclass of $N$. That is, $N$ end-extends $M$ whenever given any $a \in M$, if $N \models b \in a$ then $b \in M$.
And an end-extension $N \supseteq M$ is a {\em top-extension} or {\em rank-extension} if $M$ is a rank-initial segment of $N$.
\end{definition}

Note that top-extensions are end-extensions but the converse does not in general hold. However, if $N$ is an elementary end-extension of $M \models \ZF$ then it must be a top-extension: By elementarity, $V_\alpha^M = V_\alpha^N$ cannot have any new elements.

\begin{theorem}[Keisler \cite{keisler1974}, Shelah \cite{shelah1978}]
Any countable model of $\ZF$ has an elementary end-extension to a rather classless model.\footnote{Keisler showed this theorem under the assumption of $\diamond$ and the assumption of $\diamond$ was later eliminated by Shelah.}
\end{theorem}

If we start with countable $M \models \ZFC$ with no definable global well-order, then any rather classless $N$ which is an elementary extension of $M$ has no amenable global well-order and therefore has no $\GBC$-realization. That such $M$ exists was proved by Easton in his doctoral dissertation \cite{easton1964}.\footnote{Easton's dissertation was published in the Annals of Mathematical Logic \cite{easton1970}, but this publication omits the section where he proves this result. See \cite[section 3]{felgner1976} for a published copy of Easton's argument, or \cite{hamkins-mathoverflow2012} for a publicly available version of the argument on MathOverflow.}

For stronger theories, such as $\KM$ or $\GB + \ETR$, we cannot hope for techniques as general as in the $\GB$ case. A significant obstacle is that these theories prove $\Con(\ZF)$; consequently, a model of $\ZF$ can fail to have, say, a $\KM$-realization simply by having the wrong theory. And we can separate models of these stronger theories in a similar way; for example, if $(M,\Xcal) \models \GB + \ETR + \neg \Con(\GB + \ETR)$, then $M$ cannot be $\KM$-realizable.

For $M \models \ZF$ to be $\KM$-realizable it must satisfy the first-order consequences of $\KM$. But this condition is in fact insufficient. There is no theory whose models are all $\KM$-realizable. In fact, we need much less than $\KM$ for this.

\begin{theorem}[Krajewski {\cite[page 475]{marek-mostowski1975}}] 
There is no first-order theory all of whose models (or even just whose countable models) are $(\GB + \ETR_\omega)$-realizable. That is, given any consistent $T$ extending $\ZF$ there is a countable model of $T$ which has no $(\GB + \ETR_\omega)$-realization.
\end{theorem}

\begin{proof}
First suppose that $M$ is a model of $T$ and $\Xcal$ is an $(\GB + \ETR_\omega)$-realization for $M$. (If no model of $T$ is $(\GB + \ETR_\omega)$-realizable we are already done.) 
Consider $\Tr \in \Xcal$, the truth predicate for $M$, which exists as it can be constructed by an elementary recursion of $\omega$. By reflection, there is a club of ordinals $\alpha \in \Ord^M$ so that $(V_\alpha^M, \Tr \cap V_\alpha^M)$ thinks that $\Tr \cap V_\alpha^M$ satisfies the recursive Tarskian definition of truth. If $M$ is $\omega$-standard, meaning that $\omega^M$ is well-founded, then $\Tr$ is the real truth predicate for $M$. If $M$ is $\omega$-nonstandard, then $\Tr$ restricted to the standard formulae is truth for $M$. In either case, we get that $V_\alpha^M \subseteq M$ satisfies the full elementary diagram of $M$ and hence is elementary in $M$. Moreover, these $V_\alpha^M$ form an elementary tower which unions up to $M$. Altogether, we have seen that the first-order part of any model of $\ETR_\omega$ must be the union of an elementary tower from its $V_\alpha$s. In particular, any $\ETR_\omega$-realizable model must contain many undefinable ordinals.

On the other hand, Paris showed \cite{paris1973} that every $T$ extending $\ZF$ has a model whose every ordinal is definable without parameters. Such models are now known as \emph{Paris models}.\footnote{The name Paris model comes from \cite{enayat2005}, which contains a thorough investigation of their properties.}
The previous paragraph shows that no Paris model can be $\ETR_\omega$-realizable, and thus no first-order theory can characterize the $\ETR_\omega$-realizable models.
\end{proof}

One take-away from this is that constructing models of stronger second-order set theories is less straightforward than constructing models of $\GB$ or $\GBC$. This topic for $\ETR$ will be taken up in section~\ref{sec:etr}. For $\KM$, probably the simplest way to produce a model comes from looking at inaccessible cardinals: if $\kappa$ is inaccessible, then $(V_\kappa, V_{\kappa+1}) \models \KM$. Of course, once we have any transitive model of $\KM$ we get countable transitive models by L\"owenheim--Skolem. Indeed, $(V_\kappa, V_{\kappa+1})$ satisfies more than just $\KM$. 

\begin{definition}
The \emph{Class Collection} axiom schema asserts that if for every set there is a class satisfying some property, then there is a single class coding the ``meta-class'' consisting of a witnessing class for every set. Formally, instances of this schema take the form
\[
\forall \bar P \left[ (\forall x \exists Y\ \phi(x,Y,\bar P)) \impl (\exists C \forall x\ \exists i\  \phi(x,(C)_i,\bar P)) \right],
\]
where $(C)_i = \{ y : (i,y) \in C \}$.

The \emph{Class Choice} axiom schema strengthens Class Collection to have that the index of the witnessing class for $x$ is $x$ itself. That is, instances of Class Collection have the form
\[
\forall \bar P \left[ (\forall x \exists Y\ \phi(x,Y,\bar P)) \impl (\exists C \forall x\ \phi(x,(C)_x,\bar P)) \right].
\]
\end{definition}

\begin{proposition}
Over $\GB$, Class Choice is equivalent to Class Collection plus Global Choice.
\end{proposition}

\begin{proof}
$(\Rightarrow)$ It is immediate that Class Collection follows from Class Choice. To see that Global Choice also follows, observe that for each $x \ne \emptyset$ there is a singleton class $Y \subseteq x$. We do not need any fragment of choice beyond that afforded by first-order logic to make this conclusion. So by Class Choice there is a class $C$ so that for each $x \ne \emptyset$ we have $(C)_x \subseteq x$ is a singleton. Now define a class function $F : V \setminus \{\emptyset\} \to V$ by setting $F(x)$ to be the unique element of $(C)_x$. It is clear that such $F$ is a global choice function, the existence of such being one equivalent formulation of Global Choice.

$(\Leftarrow)$ Suppose for each $x$ there is $Y$ so that $\phi(x,Y,P)$ for some fixed parameter $P$. By Class Collection there is a class $D$ so that for each $x$ there is $i$ so that $\phi(x,(D)_i,P)$. Now define a new class $C$ by putting $(x,y)$ into $C$ if and only if $(i_x,y) \in D$ where $i_x$ is least, according to a fixed global well-order, so that $\phi(x,(D)_{i_x},P)$. Then $C$ witnesses the instance of Class Choice for $\phi(x,Y,P)$.
\end{proof}

It can be readily checked that if $\kappa$ is an inaccessible cardinal, then $(V_\kappa,V_{\kappa+1}) \models \KM$ $+$ Class Collection. If $\AC$ holds in the background universe, it will moreover satisfy Global Choice, and hence Class Choice.
Indeed, every model of $\KM$ $+$ Class Collection looks like the $(V_\kappa,V_{\kappa+1})$ of some appropriate model of set theory. 

\begin{definition} \label{def:zfmi}
Let $\ZFmi$ denote the theory whose axioms are those of $\ZFm$ (namely, the axioms of $\ZF$ without Powerset but axiomatized with Separation and Collection instead of Replacement\footnote{In the absence of Powerset, Collection is stronger than Replacement \cite{zarach1996}.})
along with the assertion that there is a largest cardinal $\kappa$ and that $\kappa$ is inaccessible. To be clear since in this context not all sets have powersets, the assertion that $\kappa$ is inaccessible is: $\kappa$ is regular, $V_\kappa$ exists, and for every $a \in V_\kappa$ we have that $\powerset(a)$ exists and is an element of $V_\kappa$. And let $\ZFCmi$ be the theory obtained from $\ZFmi$ by adding the assertion that every set can be well-ordered.
\end{definition}

Marek and Mostowski \cite{marek-mostowski1975} showed that $\KM$ $+$ Class Collection and $\ZFmi$ are bi-interpretable, as are $\KMC$ $+$ Class Collection and $\ZFCmi$. The backward direction is simple. Namely, if $M \models \ZFmi$ with largest cardinal $\kappa \in M$ then one can straightforwardly see that $({V_{\kappa}}^M,{V_{\kappa+1}}^M) \models \KM$ $+$ Class Collection. To have a name for this, call it the \emph{cutting-off construction}.

The forward interpretation goes through what I will call the \emph{unrolling construction}. The idea is, just as we can code hereditarily countable sets as subsets of $\omega$, we can code `sets' of rank $> \Ord$ via certain class relations. Call a well-founded, extensional binary relation with a maximum element a \emph{membership code}. These are the representatives of the sets in the unrolled structure, where we identify isomorphic membership codes. For example, $\mathord\in \rest \tc(\{x\})$ is a membership code which represents the set $x$. We can define a suitable membership relation between membership codes---namely $A$ is a member of $B$ if there is a predecessor $a$ to the maximum element of $B$ so that $A$ is isomorphic to the restriction of $B$ below $a$. One can then show that this structure satisfies all of $\ZFmi$. And if one carries out one interpretation then the other, the resulting model is isomorphic to the original model.\footnote{Antos and S.\ Friedman \cite{antos-friedman2017} consider a different version of this construction, where they use trees for coding rather than well-founded extensional graphs. Their construction gives the same bi-interpretability result.}

In section \ref{sec:km} I will use that the unrolling construction can be carried out over models satisfying much less than $\KM$. The key facts are the same as in the $\KM$ $+$ Class Collection case. Namely, if $(M,\Xcal)$ satisfies a strong enough theory to carry out the unrolling construction then carrying out the unrolling construction followed by the cutting construction produces a model isomorphic to $(M,\Xcal)$, and this isomorphism is definable. And the same holds for cutting off followed by unrolling. In my dissertation, I show that $\GB + \ETR$ suffices to carry out the unrolling construction to get a structure which satisfies an appreciable basic set theory \cite[chapter 2]{williams-diss}.\footnote{Global Choice is taken as a basic axiom in that chapter, but it is straightforward to see that it is not necessary to carry out the unrolling construction and prove basic facts about the unrolled structure.}
To give more detail: $\GB$ alone is enough to define the unrolled structure, with an appropriate membership relation between well-founded, extensional class relations. To see that the unrolled structure satisfies basic axioms of set theory, enough to carry out the cutting off construction, I went through the key lemma that between two well-founded, extensional class relations $A$ and $B$ there is a maximum initial partial isomorphism.\footnote{An initial partial isomorphism is a partial isomorphism whose domain and range are both downward closed.}
From this lemma one can prove that the unrolled structure satisfies Extensionality, Pairing, and other basic axioms of set theory, including $\Pi_0$-Separation and Transfinite Recursion for $\Pi_0$-properties. This lemma is an easy consequence of $\ETR$, since the maximum initial partial isomorphism is built by an elementary transfinite recursion along either $A$ or $B$. I conjecture that this lemma cannot be proven from $\GB$ alone. Specializing this lemma to the case where $A$ and $B$ are well-orders we get that class well-orders are comparable.

\begin{conjecture*}
$\GB$ does not prove the comparability of class well-orders, asserting that given any two class well-orders $\Gamma$ and $\Delta$ one embeds as an initial segment of the other.
\end{conjecture*}

To end this section, I wish to remark that when looking at the phenomenon of minimum models moving from $\KM$ to $\KM$ $+$ Class Collection is harmless. 

\begin{theorem}[{Marek and Mostowski \cite[theorem 2.5]{marek-mostowski1975}}] \label{thm:kmp-inside-km}
Let $(M,\Xcal) \models \KM$ and let $N \in \Xcal$ be an inner model of $M$. Then there is $\Ycal \subseteq \Xcal$ so that $(N,\Ycal) \models \KM$ $+$ Class Collection. If $N$ has a global well-order in $\Xcal$, then $\Ycal$ may be picked so that $(N,\Ycal)$ satisfies Global Choice, and hence also Class Choice. 
\end{theorem}

Ratajczyk showed \cite{ratajczyk1979} that an appropriate version of Marek and Mostowski's theorem holds with $\KM$ replaced by $\GB + \PCA$,\footnote{Ratajczyk refers to this theory as $\KM_1$.}
and a restricted form of Class Collection.

\section{Historical background}

Before moving to the novel content in this article, I wish to place my results in the context of previous work. 
The first result in this theme I will mention is a classical theorem, due to Marek and Mostowski \cite{marek-mostowski1975}.
They proved that every $\beta$-$\KM$-realizable model has a minimum $\beta$-$\KM$-realization.\footnote{See theorem~\ref{thm:mm-min-beta-km-rlzn} below.}
As an immediate corollary, if $M$ has a definable global well-order and is $\beta$-$\KM$-realizable then it has a minimum $\beta$-$\KMC$-realization. And it then follows from observation \ref{obs:inner-beta-models} below that if $M$ is $\beta$-$\KM$-realizable then $M$ has a minimal $\KM$-realization, which is moreover a minimal $\KMC$-realization.

Their methods were improved by Ratajczyk \cite{ratajczyk1979} to apply not just to $\KM$ but also to $\GB + \PnCA k$, for $k \ge 1$. This yields the same results from the previous paragraph but with $\KM$ replaced with $\GB + \PnCA k$.

More recently, Antos and S.\ Friedman \cite{antos-friedman2017} investigated what they call minimal $\beta$-models of $\mathsf{MK}^{**}$, a certain strengthening of $\KMC$ $+$ Class Collection. They showed that if $(M,\Xcal)$ is a countable $\beta$-model of $\mathsf{MK}^{**}$ then it can be expanded by hyperclass forcing to a $\beta$-model $(N,\Ycal)$ of $\mathsf{MK}^{**}$ containing a certain real $r$ which is minimum among $\beta$-models of $\mathsf{MK}^{**}$ containing $r$. My main theorem \ref{min-trans-km} can be strengthened to show that if $T$ is a fixed computably axiomatizable extension of $\KM$, then given any real $r$ there is never a transitive model of $T$ which is minimum among those transitive models of $T$ containing $r$; see theorem \ref{thm:main}. I view this as justifying Antos and Friedman's restriction to looking only at $\beta$-models. They justify this restriction, in their footnote 3, by analogy to $\ZFC$, saying that it is only sensible to talk about minimum models when restricting to the class of models correct about well-foundedness. In my view, this analogy is ill-founded. As other results in the current paper show, there can be minimum models in collections of models which contain non-$\beta$-models. Or for an example from first-order set theory: $L_\omegaoneck$ is the minimum transitive model of $\KP$ $+$ Infinity. But $L_\omegaoneck$ is not a $\beta$-model; Harrison showed \cite{harrison1968} there are computable linear orders which are ill-founded but have no infinite descending chain in $L_\omegaoneck$.
Nevertheless, Antos and Friedman's restriction to $\beta$-models is necessary to achieve such a result in their context.

Also important to mention in this context is that the class of transitive models of $\KM$ is strictly larger than the class of $\beta$-models of $\KM$. Indeed, Marek and Mostowski showed \cite[theorem 3.2]{marek-mostowski1975} that if $\tau$ is the least height of a transitive model of $\KM$ and $\beta$ is the least height of a $\beta$-model of $\KM$, then $\tau < \beta$ and $L_\beta \models \tau$ is countable. Applying Ratajczyk's improvements gives an analogous result for $\GB + \PnCA k$, for $k \ge 1$. 

Moving from models of second-order to a more general context, the same questions can be asked about other second-order theories. The analogies and disanalogies between the second-order set theory and second-order arithmetic have received considerable attention, much of it from the Warsaw school of logicians centered around Mostowski. Similar questions to the ones I study in this article have been asked and answered in the context of second-order arithmetic. Before I review those results, let me briefly describe the setting for second-order arithmetic.

Models of second-order arithmetic are two-sorted. One sort, the first-order part of the model, consists of the numbers, while the other, the second-order part, consists of sets of numbers. The chief axiom for the numbers is induction, and the sets are axiomatized by principles similar to those for classes in second-order set theory. This produces the following analogy between theories of arithmetic and set theories. (I omit mention of $\RCA_0$ and $\WKL_0$, the weakest of the ``Big Five'' theories of second-order arithmetic, as this paper does not study any set theoretic analogues thereof.)

\begin{center}
\begin{tabular}{c c}
Arithmetic & Set Theory \\
\hline
$\PA$ & $\ZF$ \\
$\ACA_0$ & $\GB$ \\
$\ATR_0$ & $\GB + \ETR$ \\
$\PnCA k_0$ & $\GB + \PnCA k$ \\
$\Zsf_2$ & $\KM$
\end{tabular}
\end{center}

\newcommand\ARITH{\mathrm{ARITH}}
\newcommand\HYP{\mathrm{HYP}}
In set theory, there are a plethora of different first-order models whose membership relation is well-founded. In arithmetic on the other hand, there is a unique first-order model which is well-founded, namely $\omega$. Those models of second-order arithmetic with $\omega$ as their first-order part are appropriately known as \emph{$\omega$-models}. It is common practice to identify $\omega$-models with their second-order part, speaking for example of the $\omega$-model $\HYP$ rather than $(\omega,\HYP)$. Those $\omega$-models which are moreover correct about which of their sets are well-founded are called \emph{$\beta$-models}. The following are  some of the results about minimum/minimal $\omega$- and $\beta$-models of arithmetic.\footnote{The reader who wishes to know more is advised to consult \cite{simpson:book}, especially chapter VIII. In particular, they can find therein proofs and references for the following theorems.}
Below, $\ARITH$ is the collection of arithmetical subsets of $\omega$ and $\HYP$ is the collection of hyperarithmetical subsets of $\omega$.

\begin{theorem*}[Mostowski]
$\ARITH$ is the minimum $\omega$-model of $\ACA_0$. In general, if $M \models \PA$ then $\Def(M)$, the collection of parametrically definable subsets of $M$ is the minimum $\ACA_0$-realization for $M$. 
\end{theorem*}

\begin{theorem*}[Kleene]
$\HYP$ is the minimum $\omega$-model of $\Delta^1_1$-$\mathsf{CA}_0$, whose principal axiom is $\Delta^1_1$-Comprehension.
\end{theorem*}

\begin{theorem*}[Gandy-Kreisel-Tait]
$\HYP$ is the intersection of all $\omega$-models of $\Zsf_2$. Indeed, the same holds with $\Zsf_2$ replaced with any computably axiomatizable extension of $\Sigma^1_1$-$\AC_0$ (in the language of second-order arithmetic).\footnote{See \cite[section VIII.4]{simpson:book} for a definition of $\Sigma^1_1$-$\AC_0$, whose principal axiom is the axiom of choice for $\Sigma^1_1$ definable collections of sets. The reader should also be aware that $\ATR_0$ proves every instance of the schema of $\Sigma^1_1$-Choice.}
\end{theorem*}

Combined with the fact that $(\omega,\HYP) \not \models \Sigma^1_1$-$\AC_0$, this theorem immediately implies that no computably axiomatizable extension of $\Sigma^1_1$-$\AC_0$ has a minimum $\omega$-model. In particular, $\Zsf_2$ has no minimum $\omega$-model. 

\begin{theorem*}[Simpson]
$\HYP$ is the intersection of all $\omega$-models of $\ATR_0$. And the same holds if we restrict to $\beta$-models. In particular, $\ATR_0$ has no minimum $\beta$-model.
\end{theorem*}

\begin{theorem*}[Gandy]
$\Zsf_2$ has a minimum $\beta$-model. 
\end{theorem*}

Simpson showed that the same holds for $\ATR_0$ and $\PnCA k_0$, for $k \ge 1$.

\begin{theorem*}[H.~Friedman]
There is no minimal $\omega$-model of $\Zsf_2$.
\end{theorem*}

\begin{theorem*}[Quinsey]
Let $T$ be a computably axiomatizable extension of $\ATR_0$ in the language of second-order arithmetic. Let $(M,\Xcal) \models T$ be countable. Then there is $Y \subsetneq \Xcal$ so that $(M,\Ycal) \models T$. That is, no countable $M \models \PA$ has a minimal $T$-realization. In particular, $T$ has no minimal $\omega$-model.
\end{theorem*}

Several of these results have reflections upward in the set theoretic realm. For example, that $\Zsf_2$ has no minimum $\omega$-model is the arithmetic counterpart to the fact that $\KM$ has no minimum transitive model. I want to highlight one case where there is a known disanalogy. Simpson's result that there is no minimum $\beta$-model of $\ATR_0$ does not carry over to the set theoretic realm. Main theorem \ref{min-beta-etr} is that there is a minimum $\beta$-model of $\GB + \ETR$. The underlying reason for this disanalogy is that in set theory, a class being well-founded is an elementary property, whereas in arithmetic it is a $\Pi^1_1$ property, and in fact is $\Pi^1_1$-universal. As will be made clear in section \ref{sec:etr}, my construction of the minimum $\beta$-model of $\GB + \ETR$ heavily uses that submodels of $\beta$-models of set theory with the same first-order part are automatically $\beta$-models, a fact that follows immediately from well-foundedness being elementary. From the other direction, Simpson's argument uses that being well-founded is $\Pi^1_1$-universal in arithmetic. This is the source of other disanalogies between $\GB + \ETR$ and $\ATR_0$; see e.g. \cite{sato2014} or \cite{fujimoto2012}.

\section{\texorpdfstring{$\GB$}{GB}}

Let us begin with a classical result.

\begin{theorem}[Shepherdson \cite{shepherdson1953}] \label{thm:min-gb}
If there is a transitive model of $\GB$ then there is a minimum transitive model of $\GB$, and it moreover satisfies Global Choice.
\end{theorem}

\begin{proof}
Let $L_\alpha$ be the minimum transitive model of $\ZFC$. Then $(L_\alpha, \Def(L_\alpha)) = (L_\alpha, L_{\alpha+1})$ is the minimum transitive model of $\GBC$. If $(M,\Xcal)$ is a model of $\GBC$ whose height is $\alpha$, then $L_{\alpha+1} \subseteq \Xcal$ because $L$ is a definable class. If the height of $M$ is greater than $\alpha$, then $L_{\alpha+1} \subseteq M$. Finally, note that it satisfies Global Choice because the $L$-order is definable.
\end{proof}

Essentially the same argument gives the following.

\begin{corollary} \label{cor:least-beta-gbc}
If there is a $\beta$-model of $\GB$ then there is a minimum $\beta$-model of $\GB$, and it satisfies Global Choice.
\end{corollary}

The key to proving this is the following observation.

\begin{observation} \label{obs:inner-beta-models}
\ 
\begin{enumerate}
\item Let $M \models \ZF$. If $(M,\Xcal) \models \GB$ is a $\beta$-model and $\Ycal \subseteq \Xcal$ then $(M,\Ycal)$ is a $\beta$-model. 
\item Moreover, if $(M,\Xcal) \models \GB$ is a $\beta$-model and $(N,\Ycal) \models \GB$ is such that $\Ord^N = \Ord^M$, $N \subseteq M$, and $\Ycal \subseteq \Xcal$, then $(N,\Ycal)$ is a $\beta$-model.
\end{enumerate}
\end{observation}

\begin{proof}
$(1)$ This follows immediately from the fact that ``$R$ is a well-founded class relation'' can be expressed via a first-order formula. Namely, I claim that, over $\GB$, a binary relation $R$ is well-founded if and only if every nonempty \emph{set} $d \subseteq \dom R$ has a minimal element. Clearly, every well-founded relation satisfies this property. For the other direction, suppose that $R$ is ill-founded. Let $\alpha_0$ be least so that $V_{\alpha_0}$ contains $x \in \dom R$ so that the class $R \rest x = \{ (y,y') \in R : y,y' <_R x \}$, where $<_R$ is the transitive closure of $R$, is ill-founded. Let $X_0 \subseteq V_{\alpha_0}$ be the set of such $x$. Given $\alpha_n$ and $X_n$, let $\alpha_{n+1}$ be least such that for each $x' \in X_n$ we have $V_{\alpha_{n+1}}$ contains some $x \mathbin R x'$ for which $R \rest x$ is ill-founded. And set $X_{n+1} \subseteq V_{\alpha_{n+1}}$ to be the set of these $x$. Finally, set $X = \bigcup_{n \in \omega} X_n$. Then $X$ is a set and $R \rest X$ is an ill-founded set-sized subrelation of $R$. So there is some $d \subseteq X$ without an $R$-minimal element.

$(2)$ Suppose $R \in \Ycal$ is ill-founded. Then $(M,\Xcal)$ thinks that $R$ is ill-founded, so there is $d \in M$ witnessing this. Let $\theta \in \Ord^M$ be so that $d \in V_\theta^M$. Set $r = R \cap V_\theta^N \in N$. Then $r \in M$ because $N \in \Ycal \subseteq \Xcal$. Observe that $r$ is ill-founded, as $M$ sees that $d$ witnesses the ill-foundedness of $r$. But $r$ is a set in $N$ which is a transitive model of $\ZFC$. So $N$ must think $r$ is ill-founded and hence $(N,\Ycal)$ thinks $R$ is ill-founded.
\end{proof}

\begin{proof}[Proof of corollary \ref{cor:least-beta-gbc}]
Consider $(L_\alpha,\Def(L_\alpha))$, where $\alpha$ is least such that there is a $\beta$-model of $\GB$ of height $\alpha$. Let $(M,\Xcal) \models \GB$ be a $\beta$-model with $\Ord^M = \alpha$. Then $L_\alpha \supseteq M$ and $\Def(L_\alpha) \subseteq \Xcal$. By the observation, $(L_\alpha, \Def(L_\alpha)) \models \GBC$ is a $\beta$-model. It must be contained in every $\beta$-model of $\GB$.
\end{proof}

More interesting is the question of when there is a minimum $\GBC$-realization for a model of $\ZFC$. Note that this is not interesting if we do not include Global Choice, as $\Def(M)$ is always the minimum $\GB$-realization for $M$.

\begin{theorem}[Main theorem \ref{main:min-gbc-rlzn}] \label{least-gbc}
Let $M$ be a countable model of $\ZFC$. Then $M$ has a minimum $\GBC$-realization if and only if $M$ has a parametrically definable global well-order.
\end{theorem}

\begin{proof}
$(\Leftarrow)$ Any $\GBC$-realization for $M$ must contain $\Def(M)$. If $M$ has a definable global well-order then $\Def(M)$ is itself a $\GBC$-realization, hence it is the minimum $\GBC$-realization for $M$.

$(\Rightarrow)$ Suppose towards a contradiction that $\Xcal$ is the minimum $\GBC$-realization for $M$. As $M$ has no definable global well-order, it must be that $\Xcal = \Def(M;G)$, where $G \not \in \Def(M)$ is some global well-order of $M$. Consider the forcing to add a Cohen subclass to $\Ord$ with set-sized conditions. Note that this forcing does not add any new sets as it is $\kappa$-closed for every $\kappa$. Let $H$ be $(M,\Xcal)$-generic for this forcing. That is, $H$ meets every dense class in $\Xcal$. Then every set $a \in M$, is coded into $H$. To expound, $\ZF$ proves that every set $a$ is uniquely determined by the isomorphism type of the graph $E_a = (\tc(\{a\}),\mathord\in \rest \tc(\{a\}))$. Now use Choice to take an isomorphic copy of $E_a$ on the ordinals and look at its image under your favorite $\Sigma_0$-definable pairing function $\Ord^2 \to \Ord$. This gives an ordinal-length binary sequence $s_a$ from which we can recover $a$. Indeed, $a$ is constructible from $s_a$. By density, every ordinal-length binary sequence is coded into $H$, and thereby every set is coded into $H$.  This yields a global well-order definable from $H$: say that $a <_H b$ if the index where $a$ is first coded into $H$ is less than the index where $b$ is first coded into $H$. Therefore, $\Def(M;H)$ is a $\GBC$-realization for $M$. 

I claim that $G \not \in \Def(M;H)$. Otherwise, there is a first-order formula $\phi$, possibly with set parameters and $H$ as a class parameter but where $G$ does not appear as a parameter, so that $(M,\Def(M;G,H)) \models \forall x\ x \in G \iff \phi(x,H)$. There is some $p \in H$ forcing this statement; that is, there is $p \in H$ so that $(M,\Def(M;G)) \models \text{``}p \forces \forall x\ x \in \check G \iff \phi(x,\dot H)\text{''}$. Thus, $(M,\Def(M;G)) \models \text{``} \forall x\ x \in G \iff p \forces \phi(\check x, \dot H)\text{''}$ so for all $x \in M$, we have $x \in G$ if and only if $(M,\Def(M;G)) \models [p \forces \phi(\check x, \dot H)]$. But this formula doesn't depend upon $G$, so this happens if and only if $M \models [p \forces \phi(\check x, \dot H)]$. Therefore $G \in \Def(M)$, contradicting that $M$ has no definable global well-order.

As $G \not \in \Def(M;H)$, it cannot be that $\Xcal = \Def(M;G) \subseteq \Def(M;H)$, contradicting that $\Xcal$ is the minimum realization. So $M$ cannot have a minimum $\GBC$-realization, as desired.
\end{proof}

As remarked in Section \ref{sec:prelim}, a model of $\ZFC$ having a definable global well-order is equivalent to that model satisfying the first-order sentence $\exists x\ V = \HOD(\{x\})$. Thus, whether a countable model of set theory has a minimum $\GBC$-realization is recognizable from the theory of the model. 

\begin{corollary}
A countable model of $\ZFC$ has a minimum $\GBC$-realization if and only if it satisfies $\exists x\ V = \HOD(\{x\})$. \qed
\end{corollary}

\begin{corollary}
Let $M \models \ZFC$ be countable. Then $\Xcal_\GBC(M) = \Def(M)$.\footnote{Recall definition \ref{def:xcalt}.}
\end{corollary}

\begin{proof}
If $\Def(M)$ is a $\GBC$-realization for $M$ then it is the minimum $\GBC$-realization, whence $\Xcal_\GBC(M) = M$. Otherwise, the argument from the proof of theorem \ref{least-gbc} that $G \not \in \Def(M;H)$ applies to any $A \in \Def(M;G) \setminus \Def(M)$. Thus $\Def(M;G) \cap \Def(M;H) = \Def(M)$ and so $\Xcal_\GBC(M) \subseteq \Def(M)$. The other inclusion is trivial, giving the desired equality.
\end{proof}

The construction used in the proof of theorem \ref{least-gbc} does not show show that any countable $M \not \models \exists x\ V = \HOD(\{x\})$ lacks a \emph{minimal} $\GBC$-realization. The reason is that the $\Def(M;G)$ and $\Def(M;H)$ from the proof are incomparable. We have already seen one direction, that $\Def(M;G) \not \subseteq \Def(M;H)$. The other direction, that $\Def(M;H) \not \subseteq \Def(M;G)$, is because by genericity $H \not \in \Def(M;G)$. We get that $\Def(M;G)$ is not the minimum $\GBC$-realization, but left open is the possibility that it could be minimal.

\begin{question}
Is there $M \models \ZFC + \forall x \ V \ne \HOD(\{x\})$ with a minimal $\GBC$-realization $\Xcal$ for $M$? Can this happen for countable $M$?
\end{question}

Answering this question in general seems difficult, but something can be said for special cases. First is a negative result that minimal $\GBC$-realizations cannot come from Cohen forcing.

\begin{proposition} \label{prop:cohen-not-min}
Consider $M \models \ZFC$ and suppose $C \subseteq \Ord^M$ is Cohen-generic. Then, $\Def(M;C)$ is not a minimal $\GBC$-realization.
\end{proposition}

\begin{proof}
Define $C'$ as $\alpha \in C'$ if and only if $\alpha \cdot 2 + 1 \in C$. If we think of $C$ as an $\Ord$-length binary sequence, $C'$ is what appears on the odd coordinates of $C$. It is a standard result about Cohen forcing that $C'$ is Cohen-generic but $C \not \in \Def(M;C')$. Thus, $\Def(M;C') \subsetneq \Def(M;C)$ and so $\Def(M;C)$ is not minimal.
\end{proof}

If we want a model of $\ZFC$ with a minimal $\GBC$-realization which is not minimum, we must start with a model without a definable global well-order. We want to add a strongly amenable global well-order to this model which is minimal in a certain sense. A natural place to look here is at generalizations of Sacks forcing. As is well-known, Sacks reals are minimal over the ground model. Can we generalize Sacks forcing to produce a minimal generic class of ordinals which codes a global well-order?

Generalizing Sacks forcing to add generic classes of ordinals has been considered before. Hamkins, Linetsky, and Reitz \cite{hamkins:math-tea} consider so-called ``perfect generics'', an adaptation of a technique of Kossak and Schmerl \cite{kossak-schmerl:book} from models of arithmetic. Kossak and Schmerl's technique is in turn an adaptation from Sacks forcing over $\omega$. They produce minimal generics for countable models $M$ of arithmetic, i.e.\ inductive $G \subseteq M$ so that for $A \in \Def(M;G)$ either $A \in \Def(M)$ or $G \in \Def(M;A)$.

I do not see a way to use perfect generics to produce a minimal but not minimum $\GBC$-realizations.
Nevertheless, a related result can be achieved.

\begin{theorem}[Main theorem \ref{main:min-gbc-ext}] \label{thm:min-above}
Let $(M,\Xcal) \models \GBC$ be countable. Then there is $\Ycal \supsetneq \Xcal$ a $\GBC$-realization which is minimal above $\Xcal$: if $\Zcal \supsetneq \Xcal$ is a $\GBC$-realization, then $\Ycal \subseteq \Zcal$.
\end{theorem}

This theorem shows that the poset consisting of countable $\GBC$-realizations of a fixed countable $M \models \ZFC$ is not dense in a strong sense. For any $\Xcal$ in this poset there is $\Ycal \supseteq \Xcal$ with no elements of the poset in between them. On the other hand, proposition \ref{prop:cohen-not-min} gives $\Zcal \supseteq \Xcal$ so that the interval $(\Xcal, \Zcal)$ of intermediate $\GBC$-realizations contains a dense linear order.

Before proving this theorem, it will be necessary to set up some machinery.

Work with a fixed countable $\Mfrak = (M,\Xcal) \models \GBC$. Let $\Bbb$ denote the full $\Ord$-length binary tree. For a perfect tree $\Pbb \subseteq \Bbb$, say that $\Qbb \subseteq \Pbb$ is {\em $n$-deciding for $\Pbb$} if for every $\Sigma_n$ formula $\phi$ in the forcing language for $M$, there is an ordinal $\alpha$ so that if $p \in \Qbb$ has length greater than $\alpha$ then $p$ decides $\phi$. It is not hard to see that if $\Pbb \in \Xcal$ then we can find $\Qbb \in \Xcal$ which is $n$-deciding for $\Pbb$. Moreover, we can find such $\Qbb$ which does not split below a fixed $\alpha \in \Ord$. 

\begin{definition}
A function $G : \Ord \to 2$ is called a {\em perfect generic} if there is a sequence $\Bbb = \Pbb_0 \supseteq \Pbb_1 \supseteq \cdots \supseteq \Pbb_n \supseteq \cdots$ of perfect trees from $\Xcal$ so that $G = \bigcup (\bigcap_n \Pbb_n)$ and $\Pbb_{n+1}$ is $n$-deciding for $\Pbb_n$ for all $n$.
\end{definition}

A subtle point is that a perfect generic need not be fully generic over any of the $\Pbb_n$. So the model generated by a perfect generic $G$, namely 
\[
\Mfrak[G] = (M, \{ X \subseteq M : X \in \Def(M; G, A) \text{ for some } A \in \Xcal \}),
\]
need not be a class forcing extension of $\Mfrak$.
As such, we need to refine some standard facts about class forcing. I summarize them here. First, the forcing relation (for pretame forcings, such as the posets we will use) for formulae of bounded complexity is definable. That is, the relation $p \forces_\Pbb \phi$, confined to $\phi$ which are $\Sigma_k$ is $\Sigma_j$-definable from $\Pbb$ for some $j > k$. (It does not matter for our purposes what this $j$ is, just that it is finite.) Second, we need to refine the usual ``forcing equals truth'' lemma: For $\phi$ which are $\Sigma_k$ there is $j > k$ so that for any $G \subseteq \Pbb$ which is $\Sigma_j$-generic over $\Mfrak$ then $\Mfrak[G] \models \phi$ if and only if there is $p \in G$ so that $p \forces_\Pbb \phi$. This can be proved via a standard inductive argument, attending to the complexity of the definitions of the dense sets involved. In short, the standard facts for forcing apply for partial generics, except that we must be careful to confine ourselves to things which are sufficiently simple.

Using these refined forcing facts, we can see that $\Mfrak[G]$ will still satisfy $\GBC$, despite not in general being a forcing extension. Extensionality, Global Choice, and Elementary Comprehension are immediate. To see $\Mfrak[G]$ satisfies Class Replacement, take a class function $F$. Then $F$ is definable by a $\Sigma_k$-formula for some $k$ with parameters $G$ and $A \in \Xcal$. But then the behavior of $F$ is forced by conditions in $\Pbb_{j} \supseteq G$ for large enough $j > k$.

The following lemmata will be used to prove the theorem. They are set theoretic analogues of results from section 6.5 of \cite{kossak-schmerl:book}.

\begin{lemma} \label{lem:per-gen-min}
Let $\phi(x)$ be a formula in the forcing language and $\Pbb \in \Xcal$ be a perfect subtree of $\Bbb$. There is $\Qbb \subseteq \Pbb$ in $\Xcal$ so that one of the two cases holds:
\begin{enumerate}
\item There is an ordinal $\alpha$ so that for all ordinals $\xi$ we have that all $p \in \Qbb$ of length greater than $\alpha$ decide $\phi(\xi)$ (in $\Pbb$) the same.

\item For every ordinal $\alpha$ there is $\beta > \alpha$ so that if $p,q \in \Qbb$ both have length $\beta$ and $p \rest \alpha = q \rest \alpha$ then there is an ordinal $\xi$ so that $p$ and $q$ decide $\phi(\xi)$ differently (in $\Pbb$).
\end{enumerate}
\end{lemma}

\begin{proof}
Fix $k$ so that $\phi$ is a $\Sigma_k$ formula. Take $\Pbb' \subseteq \Pbb$ a $k$-deciding subtree for $\Pbb$. We may assume that there is a function $f : \Bbb \to \Pbb'$ which embeds the full binary tree onto the splitting nodes of $\Pbb'$ and that $f(s)$ decides $\phi(\length s)$. There are two cases. The first is that there is some $s \in \Bbb$ so that for every $t, t' >_\Bbb s$ if $\length t = \length t'$ then $f(t)$ and $f(t')$ decide $\phi(\xi)$ the same for all ordinals $\xi$. In this case, set $Q = \Pbb' \rest f(s)$ and get the first conclusion in the lemma.

The second case is that this does not happen for any $s \in \Bbb$. In this case, we can inductively define a $g : \Bbb \to \Pbb'$ as follows:
\begin{itemize}
\item Set $g(0) = f(0)$.
\item Set $g(s \cat 0) = p_0$ and $g(s \cat 1) = p_1$, where $p_0, p_1$ are least (according to a fixed global well-order) so that $\length p_0 = \length p_1$ and there is an ordinal $\xi$ so that $p_0$ and $p_1$ decide $\phi(\xi)$ differently. Such $p_0$ and $p_1$ always exist, as otherwise we would be in the previous case.
\item At limit stages take unions.
\end{itemize}
Set $\Qbb = \{ p \in \Pbb' : \exists s \in \Bbb\ p \le_{\Pbb'} g(s) \}$. This yields the second conclusion in the lemma.
\end{proof}

Observe that we used global choice in an essential manner here. There are possibly many choices for $p_0$ and $p_1$ in the successor stage of the construction of $g$. In order to guarantee that $g \in \Xcal$ and hence that $\Qbb \in \Xcal$, we need to uniquely specify a choice.

This puts us in a position to prove a minimality lemma for perfect generics. 

\begin{lemma}
Let $(M,\Xcal) \models \GBC$ be countable. Then there is a perfect generic $G \not \in \Xcal$ so that for any class $A$ of ordinals definable from $G$ (and possibly parameters from $\Xcal$), either $G$ is definable from $A$ and parameters from $\Xcal$ or else $A \in \Xcal$.
\end{lemma}

\begin{proof}
Fix a cofinal sequence $\seq{\alpha_n}$ of ordinals and an enumeration $\seq{\phi_n(x)}$ of formulae in the forcing language. We construct a descending sequence of perfect trees
\[
\Bbb = \Qbb_0 \supseteq \Pbb_0 \supseteq \Qbb_1 \supseteq \Pbb_1 \supseteq \cdots \supseteq \Qbb_n \supseteq \Pbb_n \supseteq \cdots
\]
so that $\Pbb_n$ is an $n$-deciding subtree of $\Qbb_n$ which does not split below $\alpha_n$ and $\Qbb_{n+1} \subseteq \Pbb_n$ is as in the previous lemma for $\phi_n$. Our perfect generic is $G = \bigcup (\bigcap_n \Pbb_n)$. 

Consider a class of ordinals $A$ defined from $G$ and parameters in $\Xcal$, where $G \forces x \in A \iff \phi_n(x)$. Consider $\Qbb_{n+1} \subseteq \Pbb_n$. If the first case from the previous lemma holds, then $A \in \Xcal$ because $\xi \in A$ if and only if for every $p \in \Qbb_{n+1}$ the length of $p$ being sufficiently long implies that $p \forces_{\Pbb_n} \phi_n(\xi)$. If the second case of the previous lemma holds, then we can define $G$ from $A$. In this case, $p \in \bigcap_n \Pbb_n$ if and only if for every ordinal $\alpha$ there is $q >_{\Qbb_{n+1}} p$ of length greater than $\alpha$ so that $q \forces_{\Pbb_n} \phi_n(\xi) \iff \xi \in A$ for all ordinals $\xi$. From a definition of $\bigcap_n \Pbb_n$ can easily be produced a definition for $G$.
\end{proof}

\begin{proof}[Proof of theorem \ref{thm:min-above}]
We start with countable $(M,\Xcal) \models \GBC$. Let $G$ be the perfect generic  as in the minimality lemma. Set $\Ycal$ to be the classes definable from $G$ and parameters in $\Xcal$. Then, $\Ycal$ is a $\GBC$-realization for $M$. Suppose that $\Zcal$ is a $\GBC$-realization so that $\Xcal \subseteq \Zcal \subseteq \Ycal$. Consider arbitrary $A \in \Zcal$. We may assume that $A$ is a class of ordinals by mapping $A$, via a fixed bijection $V \to \Ord$ from $\Xcal$, into the ordinals. By the minimality lemma, either $G$ is definable from $A$ or else $A \in \Xcal$. Therefore, if $\Zcal \ne \Xcal$ then $G$ is definable from an element of $\Zcal$ (with parameters from $\Xcal$) and hence $\Ycal = \Zcal$.
\end{proof}

As remarked earlier, Global Choice was used essentially in the proof of lemma \ref{lem:per-gen-min}. Proving this lemma without Global Choice would yield a construction for minimal but not minimum $\GBC$-realizations. Namely, start with a countable $M \models \ZFC$ with no definable global well-order. Let $\Xcal = \Def(M)$. Then $(M,\Xcal)$ is a model of $\GB$ which does not satisfy Global Choice. Applying the theorem to $(M,\Xcal)$ would yield a $\GBC$-realization $\Ycal$ for $M$ which is minimal above $\Xcal$. But since any $\GBC$-realization must contain $\Xcal$, this would give that $\Ycal$ is a minimal $\GBC$-realization for $M$. 

Thus, the problem of constructing a minimal but not minimum $\GBC$-realization can be reduced down to the problem of proving the minimality lemma for perfect generics without using choice. A similar question can be asked for ordinary Sacks forcing.

\begin{question}
Is choice needed to prove the minimality lemma for Sacks forcing? That is, is it consistent that there are $M \models \ZF + \neg AC$, $s \subseteq \omega^M$ Sacks-generic over $M$, and $A \in M[s]$ so that $M \subsetneq M[A] \subsetneq M[s]$?
\end{question}

\section{\texorpdfstring{$\ETR$ and its fragments}{ETR and its fragments}} \label{sec:etr}

The main result of this section is that there is a smallest $\beta$-model of $\GB + \ETR$. First, a warm-up.

\begin{theorem} \label{thm:least-etr-rlzn}
If $M \models \ZF$ is $\beta$-$(\GB + \ETR)$-realizable, then $M$ has a minimum $\beta$-$(\GB + \ETR)$-realization. Moreover, if $M$ has a definable global well-order then this minimum $\beta$-$(\GB + \ETR)$-realization satisfies Global Choice.
\end{theorem}

\begin{proof}
Suppose $(M,\Ycal) \models \GB + \ETR$ is a $\beta$-model. Externally to $(M,\Ycal)$ we will construct a $(\GB + \ETR)$-realization $\Xcal \subseteq \Ycal$ for $M$ and then see that the construction for $\Xcal$ gives the same $(\GB + \ETR)$-realization no matter which $\Ycal$ we started with. 

Let $\Xcal_0 = \Def(M)$. Observe that if $M$ has a definable global well-order, $(M,\Xcal_0)$ satisfies Global Choice. Now, given $\Xcal_n$, define $\Xcal_{n+1}$ to consist of all classes of $M$ definable from elements of $\{ \Tr_\Gamma(A) : A, \Gamma \in \Xcal_n \mand \Gamma \text{ is a class well-order} \}$.\footnote{Recall that $\Tr_\Gamma(A)$ denotes the $\Gamma$-iterated truth predicate relative to the class $A$, where $\Gamma$ is a class well-order.}
A few remarks are in order to clarify why this definition is well-formed. First, note that $\Xcal_0 \subseteq \Ycal$ and thus $(M,\Xcal_0)$ is a $\beta$-model by observation \ref{obs:inner-beta-models}. Second, if $\Xcal_n \subseteq \Ycal$ then $\Xcal_{n+1} \subseteq \Ycal$. This is because $\Ycal$ must have iterated truth predicates relative to any of its classes and because $\Ycal$ and $\Xcal_n$ agree on what class relations are well-founded, both being $\beta$-models. Inductively, this yields that each $\Xcal_n \subseteq \Ycal$ is a $\beta$-model. Because each of these is a $\beta$-model, they agree with $V$ on whether $\Gamma$ is a well-order. Finally, because $M$ is transitive it has unique iterated truth predicates relative to a given parameter. 

Set $\Xcal = \bigcup_n \Xcal_n$, so that $\Xcal \subseteq \Ycal$. Because $\Xcal$ is an increasing union of $\GB$-realizations for $M$, it must be that $\Xcal$ itself is a $\GB$-realization for $M$. And if $M$ has a definable global well-order, then $(M,\Xcal)$ moreover satisfies Global Choice. To see that $(M,\Xcal)$ satisfies Elementary Transfinite Recursion, note that if $A, \Gamma \in \Xcal$ then $A,\Gamma \in \Xcal_n$ for some $n$ and thus $\Tr_\Gamma(A) \in \Xcal_{n+1} \subseteq \Xcal$.

It remains only to see that we get the same $\Xcal$ regardless of our choice of $\Ycal$. But this is immediate, since $\Ycal$ was not actually used to define the $\Xcal_n$. We only used that $\Xcal_n \subseteq \Ycal$ to conclude that $(M,\Xcal_n)$ is a $\beta$-model, and this is true for any $\Ycal$ a $\beta$-$(\GB + \ETR)$-realization for $M$.
\end{proof}

\begin{theorem}[Main Theorem \ref{min-beta-etr}]  \label{thm:least-trans-etr}
There is a minimum $\beta$-model of $\GB + \ETR$, if there is any $\beta$-model of $\GB + \ETR$. Moreover, this minimum $\beta$-model satisfies Global Choice.
\end{theorem}

\begin{proof}
First, let us see that if $M$ is $\beta$-$(\GB + \ETR)$-realizable then so is $L^M$. This is of independent interest, so I will separate it out as its own lemma.

\begin{sublemma} \label{lem:etr-in-l}
If $M$ is $\beta$-$(\GB + \ETR)$-realizable then $L^M$ is $\beta$-$(\GBC + \ETR)$-realizable.
\end{sublemma}

Before proving this, let us see with an example why merely restricting the classes to those which are subclasses of $L$ will not work. Suppose $(M,\Ycal) \models \GB + \ETR$ is transitive with $0^\sharp \in M$. Set $\Xcal = \Ycal \cap \powerset(L^M)$. Then, $(L^M,\Xcal)$ fails to satisfy Replacement. Consider $A = 0^\sharp \cup (\Ord^M \setminus \omega) \in \Xcal$. But $A \cap \omega \not \in L^M$, so $(L^M,\Xcal)$ fails to satisfy Replacement. That is, the problem is that $\Ycal$ may contain subclasses of $L^M$ which are not amenable to $L^M$. 

\begin{proof}[Proof of lemma \ref{lem:etr-in-l}]
Let $(M,\Ycal) \models \GB + \ETR$ be a $\beta$-model. Define $\Xcal \subseteq \Ycal$ to consist of all the subsets of $L^M$ which are amenable to $L^M$. That is, $X \in \Xcal$ if and only if $X \subseteq L^M$ and $X \cap a \in L^M$ for all $a \in L^M$. By observation \ref{obs:inner-beta-models} we get that $(L^M,\Xcal)$ is a $\beta$-model. So the only work is to see that $(L^M,\Xcal) \models \GBC + \ETR$. That it satisfies Extensionality and Global Choice is immediate. To see that it satisfies Class Replacement, take a class function $F \in \Xcal$ and a set $a \in L^M$. Working in $(M,\Ycal)$, pick $\alpha$ so that $F \rest a \in {V_\alpha}^M$, which exists by Class Replacement in $(M,\Ycal)$. By amenability we now get that $f = F \cap ({V_\alpha}^M \cap L^M) \in L^M$. So $F''a = f''a \in L^M$, as desired.

To see that $(L^M,\Xcal)$ satisfies Elementary Comprehension, note that this follows immediately once we know that $L^M$ satisfies Separation with parameters from $\Xcal$, since $\{x : \phi(x,\bar P)\} \cap a = \{ x \in a : \phi(x,\bar P) \}$. We check that Elementary Comprehension with parameters from $\Xcal$ holds by induction on formulae. The atomic and boolean cases are easy, but note that we use amenability to check the atomic case for the formula ``$x \in A$''. For the quantifier case, we want to see that $b = \{ x \in a : \forall y\ \phi^{L^M}(x,y) \} \in L^M$, where I have suppressed writing the parameters in $\phi$. For each $y$, let $b_y = \{ x \in a : \phi^{L^M}(x,y) \}$, which is in $L^M$ by inductive hypothesis. Then observe that $b = \{ x \in a : \forall y\ x \in b_y \}$, so $b \in L^M$. 

Finally, let us see that $(L^M,\Xcal) \models \ETR$. Consider a first-order formula $\phi(x,Y,P)$ with the parameter $P \in \Xcal$ and a class well-order $\Gamma \in \Xcal$. By $\ETR$ in $(M,\Ycal)$ there is a class $S \in \Ycal$ which satisfies $S_g = \{ x : \phi^{L^M}(x, S \rest x, P) \}$ for all $g \in \dom \Gamma$. We inductively show that $S \in \Xcal$. This induction can be carried out inside $(M,\Ycal)$ because being amenable to $L^M$ is an elementary property. The successor case follows because $(L^M,\Xcal)$ satisfies Elementary Comprehension. For the limit case, we consider $\Delta \subseteq \Gamma$ an initial segment without a maximum element and assume that $S \rest g \in \Xcal$ for all $g \in \Delta$. We want to see that $S \rest \Delta \in \Xcal$. Toward this end, fix $a \in L^M$. Now observe that
\[
(S \rest \Delta) \cap a = \left(\bigcup_{g \in \dom \Delta} S \rest g\right) \cap a = \{ x \in a : \exists g \in \dom \Delta\ x \in S \rest g \}.
\]
Because being $S \rest g$ is an elementary property in the parameters $\Delta$ and $g$, and because we have that $S \rest g \in \Xcal$ for all $g \in \dom \Delta$, by Separation applied inside $(L^M,\Xcal)$ we thereby get that $(S \rest \Delta) \cap a \in L^M$. So $S \rest \Delta \in \Xcal$, completing the argument.
\end{proof}

Let $L_\alpha$ be so that $\alpha$ is the least ordinal with $L_\alpha$ being $\beta$-$(\GB + \ETR)$-realizable. By theorem \ref{thm:least-etr-rlzn}, let $\Xcal$ be the minimum $\beta$-$(\GB + \ETR)$-realization for $L_\alpha$. Note that $(L_\alpha,\Xcal)$ satisfies Global Choice. I claim that $(L_\alpha,\Xcal)$ is the desired minimum $\beta$-model of $\GB + \ETR$. To see this, there are two cases to consider. If $(M,\Ycal) \models \GB + \ETR$ is a $\beta$-model with $\Ord^M = \alpha$, then theorem \ref{thm:least-etr-rlzn} yields that $\Xcal \subseteq \Ycal$. For the other case, if $(M,\Ycal) \models \GB + \ETR$ with $\Ord^M > \alpha$ then $L_\alpha \in M$ and thus $M$ can construct $\Xcal$ by ordinary transfinite recursion on sets. Note that this uses that that the $\Gamma$ used to build $\Xcal$ really are well-founded from the perspective of $V$, and hence also from the perspective of $M$.
\end{proof}

Essentially the same argument gives minimum $\beta$-models for $\GB + \ETR_\Gamma$. 

\begin{lemma} \label{lem:beta-etr-gamma}
If $M \models \ZF$ is $\beta$-$(\GB + \ETR_\Gamma)$-realizable for $\Gamma \in \powerset(M)$ with $\Gamma \ge \omega^\omega$ then $M$ has a minimum $\beta$-$(\GB + \ETR_\Gamma)$-realization.
\end{lemma}

The purpose of requiring $\Gamma \ge \omega^\omega$ is that this ensures $\ETR_\Gamma$ is equivalent to the existence of $\Gamma$-iterated truth predicates relative to any class. See corollary \ref{cor:etr-iff-itr}. The same applies to later theorems, but I will suppress making this comment every time.

\begin{proof}[Proof of lemma \ref{lem:beta-etr-gamma}]
This is done almost the same as the proof for theorem \ref{thm:least-etr-rlzn}. Namely, set $\Xcal_0 = \Def(M;\Gamma)$ and set $\Xcal_{n+1}$ to consist of all classes of $M$ definable from $\Tr_\Gamma(A)$ for $A \in \Xcal_n$. Finally, set $\Xcal = \bigcup_n \Xcal_n$. These are well-defined because $\Gamma$ is a well-order. Then, similar to before, $(M,\Xcal)$ is easily seen to be contained in any $\beta$-model $(M,\Ycal) \models \GB + \ETR_\Gamma$ and is itself a $\beta$-model of $\GB + \ETR_\Gamma$.
\end{proof}

\begin{theorem} \label{thm:least-beta-etr-gamma}
Let $\Gamma$ be a name for a class well-order $\ge \omega^\omega$ which is absolute to $L$. That is, $\Gamma$ is defined by some first-order formula without parameters so that any model of $\GB$ thinks the class defined by that formula is well-ordered and moreover $\Gamma^M = \Gamma^{L^M}$ . (For instance, $\Gamma$ might be $\Ord$.) If there is a $\beta$-model of $\GB + \ETR_\Gamma$ then there is a minimum $\beta$-model of $\GB + \ETR_\Gamma$. And this minimum $\beta$-model moreover satisfies Global Choice.
\end{theorem}

If there is to be a minimum $\beta$-model of $\GB + \ETR_\Gamma$ it must have a first-order part of the form $L_\alpha$. The point of requiring $\Gamma$ to be absolute to $L$ is to exclude uninteresting counterexamples like ``$\Gamma = \Ord$ if $V = L$ and $\Gamma = \omega^\omega$ otherwise''. And $\Gamma$ should be definable by a first-order formula so that different $\GB$-realizations for the same $M$ agree on what $\Gamma$ is.

\begin{proof}
This follows from lemma \ref{lem:beta-etr-gamma} plus the fact that if $M$ is $\beta$-$(\GB + \ETR_\Gamma)$-realizable then so is $L^M$. This fact can be proved similar to the argument for lemma \ref{lem:etr-in-l}. The difference is that rather than needing to consider recursions along any well-founded relation, we only consider recursions along well-founded relations of rank $\le \Gamma$.
\end{proof}

We get minimum transitive models for $\ETR_\Ord$ and other sufficiently weak fragments of $\ETR$.

\begin{lemma} \label{lem:trans-etr-ord}
If transitive $M \models \ZF$ is $(\GB + \ETR_\Ord)$-realizable then $M$ has a minimum $(\GB + \ETR_\Ord)$-realization. If $M$ has a definable global well-order then this minimum $(\GB + \ETR_\Ord)$-realization moreover satisfies Global Choice.
\end{lemma}

\begin{proof}
For the proof of lemma \ref{lem:beta-etr-gamma} the only place it was used that $M$ is $\beta$-$(\GB + \ETR_\Gamma)$-realizable was to get that $\Gamma$ really is well-founded (i.e.\ in $V$). But the $\Ord$ of a transitive model of $\GB$ is well-founded, even if it fails to be a $\beta$-model. So the exact same argument goes through. That is, the minimum $(\GB + \ETR_\Ord)$-realization for $M$ is $\Xcal = \bigcup_n \Xcal_n$ where $\Xcal_0 = \Def(M)$ and $\Xcal_{n+1} = \bigcup \{ \Def(M; \Tr_\Ord(A)) : A \in \Xcal_n \}$.
\end{proof}

\begin{theorem} \label{thm:least-trans-etr-ord}
If there is a transitive model of $\GB + \ETR_\Ord$ then there is a minimum transitive model of $\GB + \ETR_\Ord$. And this minimum model satisfies Global Choice.
\end{theorem}

\begin{proof}
Combine lemma \ref{lem:trans-etr-ord} plus the fact that $M$ being $(\GB + \ETR_\Ord)$-realizable implies $L^M$ is $(\GB + \ETR_\Ord)$-realizable. Again, this is proved much the same as lemma \ref{lem:etr-in-l}, except looking only at recursions along well-founded relations of rank $\le \Ord$. 
\end{proof}

The attentive reader will notice that the above argument applies to more than just $\Ord$. Specifically, the only fact about $\Ord$ that was used was that if $M$ is transitive then $\Ord^M$ is well-founded. The same argument goes through if $\Ord$ is replaced by a name for an ordinal or a class well-order which is well-founded in any transitive model. So the same argument goes through more generally, for example if $\Ord$ is replaced by $\omega_1$ or $\Ord^\Ord$ or $\Ord^{\Ord^\Ord} + \Ord^\Ord + \Ord$ or many other names for well-orders. 

Left open by this analysis is how far we can go. There are names $\Gamma$ for class well-orders so that there are transitive $(M,\Xcal)$ with $\Gamma^{(M,\Xcal)}$ ill-founded (i.e.\ in $V$). In particular, this holds if $\Gamma$ is defined as the least class well-order so that $L_\Gamma$ is admissible. But that definition of $\Gamma$ requires quantifying over classes. What if we only allow definitions that are first-order?

\begin{question}
If $\Gamma$ is a name for a class well-order defined by a first-order formula, must it be the case that $\Gamma^M$ is well-founded (in the ambient universe) for any transitive $M \models \ZF$?
\end{question}

To finish this section, let us see that $\ETR$ is not equivalent to $\ETR_\Gamma$ for any $\Gamma$. Moreover, we can separate fragments of $\ETR$. This establishes that theorem \ref{thm:least-beta-etr-gamma} is not redundant with theorem \ref{thm:least-trans-etr} and that theorem \ref{thm:least-trans-etr-ord} does not immediately imply that there is a minimum transitive model of $\ETR$.

First, if $\Gamma$ is sufficiently smaller than $\Delta$ then a model of $\ETR_\Gamma$ need not be a model of $\ETR_\Delta$. 

\begin{theorem} \label{thm:sep-etr-gamma}
Let $(M,\Xcal) \models \GB$ and let $\Gamma \in \Xcal$ be a well-order $\ge \omega^\omega$. Suppose that $(M,\Xcal) \models \ETR_{\Gamma \cdot \omega}$. Then, there is $\Ycal \subseteq \Xcal$ coded in $\Xcal$ so that $(M,\Ycal) \models \GB + \ETR_\Gamma$ but $(M,\Ycal) \not \models \ETR_{\Gamma \cdot \omega}$.
If $(M,\Xcal)$ moreover satisfies Global Choice, then $\Ycal$ can be picked so that $(M,\Ycal)$ also satisfies Global Choice.
\end{theorem}

\begin{proof}
Define 
\[
\Ycal = \{ X \in \Xcal : X \in \Def(M; \Tr_\Xi), \text{ where } \Xi \text{ is some initial segment of } \Gamma \cdot \omega^M \}.
\]
Then, $\Ycal$ is coded in $\Xcal$ via the code $C = \{ ((g,\phi,a),x) : (g,\phi, a \cat x) \in \Tr_{\Gamma \cdot \omega} \}$. Also, $\Ycal$ is the increasing union of the $\Def(M; \Tr_\Xi)$ for $\Xi$ initial segments of $\Gamma \cdot \omega$, each of which is a $\GB$-realization for $M$. Thus, $(M,\Ycal) \models \GB$. Next, $(M,\Ycal)$ satisfies Elementary Transfinite Recursion for recursions of rank $\le \Gamma$ because given $X \in \Ycal$ it must be that $X \in \Def(M; \Tr_\Xi)$ for some $\Xi$ an initial segment of $\Gamma \cdot \omega$ and thus $\Tr_\Gamma(X) \in \Def(M; \Tr_{\Xi + \Gamma}(X)) \subseteq \Ycal$. Finally, $(M,\Ycal) \not \models \ETR_{\Gamma \cdot \omega}$ because it does not contain $\Tr_{\Gamma \cdot \omega}$.

To prove the fact about Global Choice, do the same construction but look at iterated truth predicates relative to a fixed global well-order.
\end{proof}

This result is optimal for transitive models because $\ETR_\Gamma$ is equivalent to $\ETR_{\Gamma \cdot n}$ for any standard $n \in \omega$.\footnote{For $\omega$-nonstandard models $\ETR_\Gamma$ is not equivalent to $\ETR_{< \Gamma \cdot \omega}$. Suppose $(M,\Xcal) \models \ETR_{\Gamma \cdot \omega}$ and let $I \subseteq \omega^M$ be an initial segment closed under addition. A similar construction can be used to get $\Ycal \subseteq \Xcal$ so that $(M,\Ycal)$ satisfies $\ETR_{\Gamma \cdot n}$ if and only if $n \in I$.}

By near the same argument one can separate $\ETR_\Gamma$ from $\ETR_{<\Gamma}$, for $\Gamma$ satisfying a basic closure condition.

\begin{corollary}
Let $(M,\Xcal) \models \GB$ and let $\Gamma \in \Xcal$ be a well-order $\ge \omega^\omega$ so that $\Gamma$ is closed under addition. That is, if $\Delta < \Gamma$ is a well-order then $\Delta + \Delta < \Gamma$. Then, if $(M,\Xcal) \models \ETR_\Gamma$ there is $\Ycal \subseteq \Xcal$ coded in $\Xcal$ so that $(M,\Ycal) \models \GB + \ETR_{<\Gamma}$ but $(M,\Ycal) \not \models \ETR_\Gamma$. If $(M,\Xcal)$ moreover satisfies Global Choice then $\Ycal$ may be fixed so that $(M,\Ycal)$ also satisfies global choice. \qed
\end{corollary}

As an immediate corollary we get that fragments of $\ETR$ can be separated by consistency strength.

\begin{corollary} 
Let $\Gamma$ be a name for a well-order $\ge \omega^\omega$, given by a first-order formula without parameters. Then $\GB + \ETR$ proves $\Con(\GB + \ETR_\Gamma)$. Moreover if $\Gamma$ and $\Delta$ are names for well-orders given by first-order formulae without parameters, then $\GB$ proves $\Gamma \cdot \omega \le \Delta$ implies that $\GB + \ETR_\Delta$ proves $\Con(\GB + \ETR_\Gamma)$.
\end{corollary}

\begin{proof}
Let $(M,\Xcal) \models \GB + \ETR_\Delta$. By the previous theorem there is $\Ycal \subseteq \Xcal$ coded in $\Xcal$ so that $(M,\Ycal) \models \GB + \ETR_\Gamma$. But $(M,\Xcal)$ has access to the full second-order truth predicate for coded models, via an elementary recursion of length $\omega$. So $(M,\Xcal) \models \Con(\GB + \ETR_\Gamma)$.
\end{proof}

One can similarly state and prove a separation of $\ETR_{<\Gamma}$ and $\ETR_\Gamma$ by consistency strength. I leave this to the reader.

In theorem \ref{thm:sep-etr-gamma} we produced a model $(M,\Ycal)$ of $\GB + \ETR_\Gamma$ which was not a model of $\GB + \ETR_{\Gamma \cdot \omega}$. But this was due to a deficiency the second-order part. We produced $\Ycal$ by taking a restricted collection of classes from $\Xcal$ so as to not have long enough iterated truth predicates to get a model of $\ETR_{\Gamma \cdot \omega}$. Can we separate fragments of $\ETR$ with the first-order part of a model? That is, is there $M \models \ZFC$ so that $M$ is $(\GB + \ETR_\Gamma)$-realizable but not $(\GB + \ETR_{\Gamma \cdot \omega})$-realizable?

Of course, if we put no restriction on $M$ then this follows immediately from the fact that $\GB + \ETR_{\Gamma \cdot \omega}$ proves the consistency of $\GB + \ETR_\Gamma$. If $M$ is the first-order part of a model of $\GB + \ETR_\Gamma + \neg \Con(\GB + \ETR_\Gamma)$ then $M$ will not be $(\GB + \ETR_{\Gamma \cdot \omega})$-realizable. Such $M$ will necessarily be $\omega$-nonstandard. Can we make the same separation, but with an $M$ which is $\omega$-standard, or even transitive?

The answer is yes, looking below $\ETR_\Ord$.

\begin{theorem} \label{thm:trans-sep-etr-gamma}
Suppose $\gamma$ is a name for an ordinal $\ge \omega^\omega$ given by a first-order formula without parameters. (In particular, $\gamma$ is a set, not a proper class.) If there is a transitive model of $\GB + \ETR_{\gamma \cdot \omega}$ then there is transitive $M$ which is $(\GB + \ETR_\gamma)$-realizable but not $(\GB + \ETR_{\gamma \cdot \omega})$-realizable. 
\end{theorem}

\begin{proof}
Take $(N,\Ycal) \models \GB + \ETR_{\gamma \cdot \omega}$ which has a first-order definable global well-order. Consider the iterated truth predicates $\vec T = \{ \Tr_\xi : \xi = \gamma \cdot n \}$ and let $(M; \vec S) = \Sk^{(N, \vec T)}(\gamma)$ be the Skolem hull of $\gamma$ in the structure $(N; \vec T)$. By elementarity, $\vec S = \{ (\Tr_\xi)^M : \xi = \gamma \cdot n \}$. Moreover, $\gamma \cdot n \in N$ for all $n$ because $(N; \vec T)$ satisfies that the length of each truth predicate in $\vec T$ is a set. 

Let $\Xcal = \Def(M; \vec S)$. Then, $(M, \Xcal) \models \ETR_\gamma$; given $A \in \Xcal$ there is $n$ so that $A$ is definable from $\Tr_{\gamma \cdot n}(\bar G)$. But then $\Tr_\gamma(A)$ is definable from $\Tr_{\gamma \cdot (n+1)}$.

However, $M$ is not $(\GB + \ETR_{\gamma \cdot \omega})$-realizable. This is because the $(\gamma \cdot \omega+1)$-iterated truth predicate for $M$ reveals that $M$ is a Skolem hull and hence countable. But no model of $\ZFC$ thinks that it is countable. So no $\GB$-realization for $M$ can contain $\Tr_{\gamma \cdot \omega + 1}^M$, which $\GB + \ETR_{\gamma \cdot \omega}$ proves exists.
\end{proof}

The proof that there is a minimum $\beta$-model of $\GB + \ETR$ came down to two facts: first, that if a model of $\ZFC$ with a definable global well-order is $\beta$-$(\GB + \ETR)$-realizable then it has a minimum $\beta$-$(\GB + \ETR)$-realization; second, that if $M$ is $\beta$-$(\GB + \ETR)$-realizable then so is $L^M$. One strategy to try to prove the existence of a minimum transitive model of $\GB + \ETR$ would be to prove analogues of these two facts. The latter---i.e.\ that if $M$ is $(\GB + \ETR)$-realizable then so is $L^M$---can be proven by an argument similar to the proof of lemma \ref{lem:etr-in-l}.

I do not see how to settle the analogue of the first fact. Working towards a positive answer, the proof of theorem \ref{thm:least-trans-etr} does not generalize to show that transitive $(\GB + \ETR)$-realizable models have minimum $(\GB + \ETR)$-realizations. Working towards a negative answer, the argument in the next section that no countable model of $\ZFC$ has a minimum $\KM$-realization does not generalize down to $\GB + \ETR$. The trouble is that $\GB + \ETR$ is too weak for the same proof to go through. Namely, $\GB + \ETR$ does not prove that $\Hyp(V)$ exists; cf.~ lemma~\ref{lem:hypv-exists}.

\begin{question}
If $M$ is $(\GB + \ETR)$-realizable, must $M$ have a minimum $(\GB + \ETR)$-realization? What if we restrict to the case where $M$ is transitive?
\end{question}

A positive answer to this question would imply a positive answer to the following, the main open question from this article.

\begin{question}
Is there a minimum transitive model of $\GB + \ETR$?
\end{question}

I have a partial answer.

\begin{theorem*}[{\cite[theorem 4.40]{williams-diss}}]
Suppose $M \models \ZF$ is $(\GB + \ETR)$-realizable. Then $M$ has minimal $(\GB + \ETR)$-realizations. Further, if $\Xcal$ is any $(\GB + \ETR)$-realization for $M$ then $\Xcal$ contains precisely one of these minimal realizations. Even further, if $\Xcal$ and $\Ycal$ are amalgamable $(\GB + \ETR)$-realizations, meaning there is a $\GB$-realization $\Zcal \supseteq \Xcal,\Ycal$ for $M$, then $\Xcal$ and $\Ycal$ contain the same minimal realization.\footnote{The version of this theorem in my dissertation is stated for $M$ satisfying $\AC$, but it is straightforward to see that this extra assumption is not used.}
\end{theorem*}

\section{\texorpdfstring{$\KM$ and other strong theories}{KM and other strong theories}} \label{sec:km}

Before presenting the results about models of $\KM$, let me review for the benefit of the reader some facts from the study of admissible sets. Kripke--Platek set theory $\KP$ is the (first-order) set theory consisting of the axioms of Extensionality, Pairing, and Union, along with the $\Sigma_0$-Separation, $\Sigma_0$-Collection, and $\Pi_1$-Foundation\footnote{Recall that $\Gamma$-Foundation asserts that every nonempty class definable from a formula in $\Gamma$ has an $\in$-minimal element.} 
schemata. It will be convenient in this context to go beyond pure sets and consider \emph{urelements}, objects which are not sets but can be members of sets. $\KPU$ is the variation of $\KP$ formulated in a context allowing urelements; see \cite[Chapter I.2]{barwise1975} for a precise axiomatization. Transitive models of $\KPU$ are known as \emph{admissible sets}. I will write $\Afrak = (A,U;R_0, \ldots)$ for the admissible set $A$ with $U$ as its set of urelements, $\in$ for its membership relation, and additional relations, constants, or functions $R_0, \ldots$ Given a structure $(U;R_0, \ldots)$, where we may assume $U$ consists of urelements, there is a smallest admissible set $\Afrak = (A,U;R_0,\ldots)$ with $U \in A$. This smallest admissible set containing $U$ is denoted $\Hyp(U)$. If $U$ is a well-founded model of set theory without urelements then $U$ will be isomorphic to a pure transitive set in $\Hyp(U)$, where a set is pure if it has no urelements in its transitive closure. If we call this transitive set $\bar U$, then the pure sets of $\Hyp(U)$ are precisely $L_\alpha(\bar U)$, where $\alpha$ is the least ordinal so that $L_\alpha(\bar U) \models \KP$. And from $L_\alpha(\bar U)$ we can recover $U$ thought of as a set of urelements and all of $\Hyp(U)$. When working with $U$ a transitive model of set theory I will tacitly identify $\Hyp(U)$, the model with urelements, with $L_\alpha(\bar U)$, the model with only pure sets.

There is a tight connection between admissible sets and infinitary logics. Countable admissible sets capture fragments of $\Lcal_{\omega_1,\omega}$ which admit a form of compactness. Let $\Afrak$ be a countable admissible set. Then $\Lcal_\Afrak$ is the infinitary fragment of $\Lcal_{\infty,\omega}$ associated with $\Afrak$. Namely, formulae in $\Lcal_\Afrak$ may only have finite quantifier blocks but allow disjunctions and conjunctions over $\Phi$ a set of formulae so that $\Phi \in \Afrak$. Barwise proved that these admissible fragments of $\Lcal_{\omega_1,\omega}$ satisfy a form of compactness: if an $\Lcal_\Afrak$ theory $T \subseteq A$ is $\Sigma_1$ definable over $\Afrak$ then $T$ has a model if and only if every $T_0 \subseteq T$ in $\Afrak$ has a model. 

I will make use of the following result about admissible sets, due also to Barwise. First, let me fix some notation. Consider the structure $\Afrak = (A,U;E,\ldots)$ where $U$ consists of urelements, $A$ consists of sets, and $E$ is a binary relation. For $a \in A$ set $a_E = \{ x \in \Afrak : x \mathbin E a \}$. Say that $S \subseteq U$ is \emph{internal} for $\Afrak$ if there is $a \in A$ so that $S = a_E$. 

\begin{theorem}[Barwise {\cite[Theorem IV.1.1]{barwise1975}}] \label{thm:barwise-hard-core}
Let $(U;R_0,\ldots)$ be a countable structure and let $\Afrak = \Hyp(U)$. Suppose $T$ is an $\Lcal_\Afrak$ theory which is $\Sigma_1$-definable over $\Afrak$ and has a model of the form $\Bfrak = (B,U;E,R_0,\ldots)$. Suppose $S \subseteq U$ is internal for every such $\Bfrak$. Then $S \in \Hyp(U)$. 
\end{theorem}

This result is an infinitary version of a well-known theorem due to Gandy, Kreisel, and Tait, with $\Hyp(U)$ being the higher-order analogue to the collection of hyperarithmetic sets.

\begin{theorem*}[Gandy--Kreisel--Tait \cite{GTK1960}]
Let $T$ be a $\Pi^1_1$-definable set of axioms for second-order arithmetic, where $T$ extends $\Sigma^1_1\text{-}\axiom{AC}_0$. If $S \subseteq \omega$ is internal to every model of $T$, then $S$ is hyperarithmetic.
\end{theorem*}

The admissible preliminaries out of the way, let us begin with Marek and Mostowski's theorem that there is a smallest $\beta$-model of $\KM$. I include a proof for the sake of completeness of presentation.

\begin{theorem}[Marek--Mostowski \cite{marek-mostowski1975}] \label{thm:least-beta-km}
There is a minimum $\beta$-model of $\KM$, if there is any $\beta$-model of $\KM$. Moreover, this minimum model also satisfies Global Choice and Class Choice. 
\end{theorem}

\begin{proof}
Suppose there is a $\beta$-model of $\KM$. By theorem \ref{thm:kmp-inside-km} there is a $\beta$-model of $\KMC$ $+$ Class Collection. Its unrolling is a well-founded model of $\ZFCmi$.\footnote{Recall definition~\ref{def:zfmi}.}
So there is a transitive model of $\ZFCmi$. 
Let $L_\alpha$ be the minimum transitive model of $\ZFCmi$.\footnote{To see such exists: Take $M \models \ZFCmi$ transitive with $\kappa \in M$ the largest cardinal. Then $\kappa$ is inaccessible in $L^M$ so $L_{(\kappa^+)^{L^M}} \models \ZFCmi$. So the minimum transitive model of $\ZFCmi$ is $L_\alpha$ for $\alpha$ least such that there is a transitive model of $\ZFCmi$ of height $\alpha$.}
Let $\kappa < \alpha$ be the largest cardinal in $L_\alpha$. Set $(M,\Xcal) = (L_\kappa, \powerset(L_\kappa) \cap L_\alpha)$. I claim that $(M,\Xcal)$ is the minimum $\beta$-model of $\KM$. First, observe that it is itself a $\beta$-model since $L_\alpha$ is correct about well-foundedness. 

Take $(N,\bar \Ycal)$ a $\beta$-model of $\KM$. By theorem \ref{thm:kmp-inside-km} there is $\Ycal \subseteq \Ycal$ so that $(L^N,\Ycal) \models \KMC$ $+$ Class Collection. Now let $W \models \ZFCmi$ be the unrolling of $(N,\Ycal)$. Because $(N,\Ycal)$ is a $\beta$-model, $W$ must be well-founded. Identify $W$ with its transitive collapse. Then, $L_\alpha \subseteq W$, because $L_\alpha$ is the minimum transitive model of $\ZFCmi$. 

Let $\lambda$ be the largest cardinal in $W$. There are two cases. The first is if $\lambda > \kappa$. Then, because $W$ thinks $\lambda$ is inaccessible, $W$ thinks $\powerset(L_\kappa)$ exists. Therefore, $\Xcal = \powerset(L_\kappa) \cap L_\alpha \in {V_\lambda}^W = N$. 
The second case is if $\lambda = \kappa$. Then, $\Ord^W \ge \alpha$ by the minimality of $L_\alpha$. Thus, $\Xcal = \powerset(L_\kappa) \cap L_\alpha \subseteq \powerset(L_\kappa) \cap W$ so $\Xcal \subseteq \Ycal$. In either case, we then conclude that $M \subseteq N$ and $\Xcal \subseteq \bar \Ycal$.
\end{proof}

Essentially the same argument gives the existence of minimum $\beta$-$\KM$-realizations.

\begin{theorem}[Marek and Mostowski \cite{marek-mostowski1975}] \label{thm:mm-min-beta-km-rlzn}
Suppose $M \models \ZF$ is $\beta$-$\KM$-realizable. Then $M$ has a minimum $\beta$-$\KM$-realization. In particular, in this case $(M,\Xcal_{\beta,\KM}) \models \KM$.\footnote{Recall definition \ref{def:xcalt}.} \qed
\end{theorem}

Ratajczyk improved their methods to apply to weaker theories.

\begin{theorem}[Ratajczyk \cite{ratajczyk1979}]
Let $k \ge 1$. There is a minimum $\beta$-model of $\GB + \PnCA k$, if there is any $\beta$-model of $\GB + \PnCA k$. Moreover, this minimum model also satisfies Global Choice and Class Choice. And if $M \models \ZF$ is $\beta$-$(\GB + \PnCA k)$-realizable then $M$ has a minimum $\beta$-$(\GB + \PnCA k)$-realization.
\end{theorem}

The analogous result fails if we look at transitive models, rather than confining ourselves to $\beta$-models.

\begin{theorem}[Main Theorem \ref{min-trans-km}] \label{thm:main}
There is not a minimum transitive model of $\KM$. Moreover, for any real $r$ there is not a minimum transitive model of $\KM$ which contains $r$. Indeed, the same holds if $\KM$ is replaced by any $T$ extending $\GB + \PCA$ with the property that $T$ is an element of any transitive model of $T$. In particular, this holds for any computably axiomatizable extension of $\GB + \PCA$.
\end{theorem}

The original proof I found for this went through Barwise's notion of the \emph{admissible cover}. Ali Enayat privately communicated to me an alternative argument which goes through the above theorem of Barwise. With his kind permission it is his proof I present here.\footnote{The reader who wishes to see the admissible cover argument can find it in my dissertation \cite{williams-diss}, written while the current article was under review.}

The key fact to establish this theorem is encapsulated by the following lemma.

\begin{lemma}[Key Lemma; Ali Enayat] \label{lem:key}
Suppose $M \models \ZF$ is countable. Let $T \subseteq \omega$ be a second-order set theory extending $\GB + \DCA$ with one of the following properties, depending on whether $M$ is $\omega$-standard. If $M$ is $\omega$-standard, then $T \in M$; or if $M$ is $\omega$-nonstandard, then there is $t \in \powerset(\omega)^M$ so that $T$ is the standard part of $t$. Suppose $M$ is $T$-realizable. Then
\[
\Xcal_T(M) = \bigcap\{ \Xcal \subseteq \powerset(M) : (M,\Xcal) \models T \} = \Hyp(M) \cap \powerset(M).
\]
\end{lemma}

\begin{proof}
$(\subseteq)$ We want to consider two-sorted models of the form $\Bfrak = (B,M;E)$. This argument now breaks into cases depending upon whether $M$ is $\omega$-standard. First work in the case where $M$ is $\omega$-standard. Consider the following $\Lcal_{\Hyp(M)}$-theory $\bar T$ which can hold of such a two-sorted model. Namely, $\bar T$ asserts that $M$ and $C$ make up, respectively, the sets and classes of a model of $T$, using the membership relation $E$ so that $E \rest M$ is $\in^M$. It is immediate that $(B,M;E) \models \bar T$. Note $M$ is definable by a single $\Lcal_{\Hyp(M)}$ sentence, and so we get that $\bar T$ really can express that the first-order part of the model is $M$. Similarly, we can express as a single $\Lcal_{\Hyp(M)}$ assertion that $E \rest M$ is $\in^M$. Combined with the assumption that $T \in M$ we get that $\bar T \in \Hyp(M)$, and in particular is $\Sigma_1$-definable over $\Hyp(M)$. Observe that any $(M,\Xcal) \models T$ gives rise to $\Bfrak \models \bar T$ in the obvious manner. So if $S \subseteq M$ is an element of every $T$-realization for $M$ then $S$ is internal for every $\Bfrak \models \bar T$. So by Barwise's theorem \ref{thm:barwise-hard-core} we get that $S \in \Hyp(M)$. 

Now let us consider the case when $M$ is $\omega$-nonstandard. In this case, we want $\bar T$ to be the $\Lcal_{\Hyp(M)}$-theory asserting that $M$ and $C$ make up, respectively, the sets and classes of a model of the standard part of $t$, using the membership relation $E$ so that $E \rest M$ is $\in^M$. Again, as a single $\Lcal_{\Hyp(M)}$-sentence we can express that the first-order part really is $M$ and that $E \rest M$ is $\in^M$. So to see that $\bar T \in \Hyp(M)$, and thus that $\bar T$ really is a $\Lcal_{\Hyp(M)}$-theory, we only have to see that the standard part of $t$ is in $\Hyp(M)$. First, observe that since $\Hyp(M)$ has an infinite set as an element, namely $M$ itself, $\Hyp(M)$ must contain $\omega$ as an element. Therefore, $\Hyp(M)$ can see the unique embedding of $\omega$ onto an initial segment of $\omega^M$. From this embedding $\Hyp(M)$ can extract the standard part of any $x \in M$ so that $M \models x \subseteq \omega^M$. The rest of the argument is now the same: If $S \subseteq M$ is an element of every $T$-realization for $M$ then $S$ is internal for every $\Bfrak \models \bar T$ and so by Barwise's theorem \ref{thm:barwise-hard-core} we get that $S \in \Hyp(M)$.

$(\supseteq)$ This argument is the same whether or not $M$ is $\omega$-standard. Take $S \in \Hyp(M) \cap \powerset(M)$. Then $S$ is $\Delta^1_1$ over $M$ without class parameters, but possibly allowing parameters from $M$; see \cite[corollary IV.3.4]{barwise1975}. Observe that $\Sigma^1_1$ definitions (without class parameters) over $M$ are upward absolute. That is, for $a \in M$ and $\Xcal \subseteq \powerset(M)$, if $\phi(x)$ is $\Sigma^1_1$ then $(M,\Xcal) \models \phi(a)$ implies $(M,\powerset(M)) \models \phi(a)$. This is because any witness in $\Xcal$ to $\phi(a)$ must also be in $\powerset(M)$. And dually, $\Pi^1_1$ definitions (without class parameters) over $M$ are downward absolute; if $\phi(x)$ is $\Pi^1_1$ then $(M,\powerset(M)) \models \phi(a)$ implies $(M,\Xcal) \models \phi(a)$.
Let $\Xcal$ be any $T$-realization for $M$. Then, combining the absoluteness facts with the fact that $T$ includes $\Delta^1_1$-Comprehension we get that $S \in \Xcal$, completing the argument.
\end{proof}

\begin{proof}[Proof of theorem \ref{thm:main} (Ali Enayat)]
Fix a real $r$. If $T$ has no transitive model containing $r$, then the conclusion is trivially satisfied. So assume we are in the case where there is a transitive model of $T$ containing $r$. Suppose toward a contradiction that $(M,\Xcal)$ is the minimum transitive model of $T$ which contains $r$. By a L\"owenheim--Skolem argument it must be that $M$ is countable. So by the key lemma we get that $\Xcal = \Hyp(M) \cap \powerset(M)$. Let us see that this is in fact impossible, giving the desired contradiction.

Recall from the discussion at the end of section \ref{sec:prelim} that $T$ is strong enough to carry out the unrolling construction, working with membership codes for `sets' of rank $> \Ord$. As I remarked, the unrolling and cutting off constructions are inverses. Applied to this case, we get that cutting off $\Hyp(M)$ to get $(M,\Xcal)$ and then unrolling produces an isomorphic copy of $\Hyp(M)$. In particular, we get that $(M,\Xcal)$ unrolls to $\Hyp(M)$. 

In that earlier discussion we were interested in the structure obtained by looking at all of these membership codes, but here we are also interested in properties of individual membership codes. Specifically, given a membership code which represents a transitive set\footnote{To be clear, $E$ is a membership code for a transitive set if $a \mathbin E b \mathbin E m_E$, where $m_E$ is the maximum element of $E$, implies $a \mathbin E m_E$.}
we can ask which set theoretic axioms it satisfies. We can sensibly formulate this, since given the truth predicate $\Tr(E)$ we can extract a truth predicate for the transitive set coded by $E$. More specifically, we can ask whether $E$ represents a transitive model of $\KP$. And we can also ask whether the canonical membership code representing $M$ is a member of $E$. That is, we can query whether there is $a \mathbin E m_E$ so that $E$ restricted to below $a$ is isomorphic to $\in^M$. In summary, we can ask whether there is coded in $\Xcal$ an admissible set which contains $M$ as an element.

Let ``$\Hyp(V)$ exists'' be the statement in second-order set theory asserting that there is a membership code for a transitive model of $\KP$ which contains the class of all sets as an element. If there is any such membership code, there there is a minimum up to isomorphism such membership code. This is then the membership code for $\Hyp(V)$.

\begin{sublemma} \label{lem:hypv-exists}
$\GB + \PCA$ proves that $\Hyp(V)$ exists.
\end{sublemma}

This lemma can also be found as \cite[theorem 64]{antos-barton-friedman2018}. Antos, Barton, and Friedman also show that the classes coded as elements of $\Hyp(V)$ form a model of $\GB + \DCA$, whence they conclude that $\GB + \DCA$ does not prove that $\Hyp(V)$ exists. One can check that this model also satisfies $\ETR$, and so $\GB + \DCA + \ETR$ does not prove that $\Hyp(V)$ exists.
For the benefit of the reader, I provide a proof of this result here from the assumption of $\KM$. The reader who wants a proof from the weaker theory should consult Antos, Barton, and Friedman, or see my dissertation \cite[lemma 4.24]{williams-diss}.

\begin{proof}[Proof from $\KM$]
By theorem \ref{thm:kmp-inside-km} we may assume Class Collection without loss. Now work in the unrolling, which satisfies $\ZFmi$. Let $\kappa$ denote the largest cardinal in the unrolled structure, so that $V_\kappa$ in the unrolling is the $V$ of the ground model. Note that $\ZFm$ proves the reflection schema relative to the $L(A)$ hierarchy for any $A$. That is, for each formula $\phi(x,a)$ with a parameter $a$, $\ZFm$ proves that there are cofinally many ordinals $\alpha$ so that for all $x \in L_\alpha(A)$ we have $L_\alpha(A) \models \phi(x,a)$ if and only if $\phi(x,a)^{L(A)}$. This is proved via the usual argument, but using the $L(A)$ hierarchy instead of the $V$ hierarchy. Since $\ZFm$ includes $\KP$ and ``$N \models \KP$'' can be expressed as a single formula in the language of set theory, we get that in the unrolling there is an ordinal $\alpha > \Ord^M$ so that $L_\alpha(V_\kappa) \models \KP$. Thus, in the ground model satisfying $\KM$ $+$ Class Collection we get that there is a membership code for a transitive model of $\KP$ which contains the class of all sets as an element.
\end{proof}

By this lemma, there is a membership code which $(M,\Xcal)$ thinks represents $\Hyp(M)$ inside $\Xcal = \Hyp(M) \cap \powerset(M)$. Thus, if we unroll $\Xcal$ to get $\Hyp(M)$ then we find that there is $A \in \Hyp(M)$ which is a transitive model of $\KP$ with $M \in A$. This contradicts that $\Hyp(M)$ is the minimum such transitive model of $\KP$. Therefore, our original assumption that $(M,\Xcal)$ is the minimum transitive model of $T$ must be wrong. So $T$ does not have a minimum transitive model. 
\end{proof}

Note that we actually proved the following stronger result.

\begin{corollary}
Let $M \models \ZF$ be countable and transitive and suppose $T \in M$ is an extension of $\GB + \PCA$. Then $M$ does not have a minimum $T$-realization.
\end{corollary}

\begin{proof}
Suppose towards a contradiction that $\Xcal$ is the minimum $T$-realization for $M$. By lemma~\ref{lem:key}, it must be that $\Xcal = \Hyp(M) \cap \powerset(M)$. But because $T$ proves $\Hyp(V)$ exists we get $A \in \Hyp(M)$ which is an admissible set containing $M$ as an element, a contradiction.
\end{proof}

This result also holds in case $M$ is not transitive. (But note that if $M$ is $\omega$-nonstandard then instead of asking that $T \in M$ we have to ask that $T$ is the standard part of an $M$-real $t \in \powerset(\omega)^M$.) See \cite[corollary 4.26]{williams-diss}.

A natural next question is to weaken minimum to minimal, and ask whether there are minimal transitive models of $\KM$. This is answered by the following theorem.\footnote{I thank the referee for pointing out to me the forward direction of the argument for this theorem.}

\begin{theorem} \label{thm:when-min-rlzn}
Let $M \models \ZF$ be countable. Then $M$ has a minimal $\KM$-realization if and only if $M$ has a $\beta$-$\KM$-realization.
\end{theorem}

\begin{proof}
$(\Leftarrow)$ By theorem~\ref{thm:mm-min-beta-km-rlzn} if $M$ has a $\beta$-$\KM$-realization it must have a minimum $\beta$-$\KM$-realization. Then by observation~\ref{obs:inner-beta-models} this minimum $\beta$-$\KM$-realization must be a minimal $\KM$-realization for $M$.

$(\Rightarrow)$ This direction goes through a theorem due to Quinsey.

\begin{theorem*}[Quinsey {\cite[corollary 6.7]{quinsey1980}}]
Fix $n \in \omega$ and let $W \models \ZFm$ be countable and nonstandard. Then there is $U \models \ZFm$ so that $U \prec_{\Sigma_n} W$ is strictly end-extended by $W$.
\end{theorem*}

The argument goes by the contrapositive. Suppose that $M$ has no $\beta$-$\KM$-realization. If $M$ is not $\KM$-realizable at all, then the conclusion is trivial. For the substantive case, suppose $(M,\Xcal) \models \KM$. If $\Xcal$ were uncountable, then by the downward L\"owenheim--Skolem theorem it could not be a minimal $\KM$-realization for $M$. In the case where $\Xcal$ is countable, let $W \models \ZFmi$ be the unrolling of $(M,\Xcal)$. Then $W$ is nonstandard, as $(M,\Xcal)$ is wrong about well-foundedness. Because being the largest cardinal of a model of $\ZFm$ is a definable property, Quinsey's theorem gives us $U \models \ZFmi$ which is strictly end-extended by $W$ so that $U$ and $W$ agree on $V_\kappa$, with $\kappa$ the largest cardinal of the model. Thus, $(M,\Ycal) = ({V_\kappa}^U,{V_{\kappa+1}}^U) \models \KM$ is a strict submodel of $(M,\Xcal)$. Therefore, $\Xcal$ is not a minimal $\KM$-realization for $M$, as desired.
\end{proof}

\begin{corollary}
There is no minimal transitive model of $\KM$. That is, if $(M,\Xcal) \models \KM$ then there is $(N,\Ycal) \models \KM$ a strict submodel of $(M,\Xcal)$. 
\end{corollary}

\begin{proof}
If $(M,\Xcal)$ were to be a minimal transitive model of $\KM$, then $\Xcal$ would have to be a minimal $\KM$-realization for $M$. In such a case, by the previous theorem it must be that $(M,\Xcal)$ is a $\beta$-model. Now apply Marek and Mostowski's result \cite[theorem 3.2]{marek-mostowski1975} that every $\beta$-model of $\KM$ contains a transitive model of $\KM$ as a set. This then gives that there is $(N,\Ycal) \in M$ a transitive model of $\KM$, witnessing that $(M,\Xcal)$ cannot be a minimal transitive model of $\KM$.
\end{proof}

\printbibliography

\end{document}